\documentclass[10pt]{amsart}

\usepackage[mathscr]{euscript}
\DeclareFontFamily{OT1}{rsfs}{}
\DeclareFontShape{OT1}{rsfs}{n}{it}{<-> rsfs10}{}
\DeclareMathAlphabet{\curly}{OT1}{rsfs}{n}{it}

\makeatletter
\newcommand{\eqnum}{\refstepcounter{equation}\textup{\tagform@{\theequation}}}
\makeatother

\newcommand\beq[1]{\begin{equation}\label{#1}}
	\newcommand\eeq{\end{equation}}
\newcommand\beqa{\begin{eqnarray*}}
	\newcommand\eeqa{\end{eqnarray*}}

\title[Categorical DT theory]{Categorical Donaldson-Thomas theory 
for local surfaces: $\mathbb{Z}/2$-periodic version}
\date{}
\author{Yukinobu Toda}

\usepackage{amscd}
\usepackage{amsmath}
\usepackage{amssymb}
\usepackage{amsthm}
\usepackage{float}
\usepackage[dvips]{graphicx}
\usepackage{xypic}

\usepackage{array}
\usepackage{amscd}
\usepackage[all]{xy}

\usepackage{tabulary}
\usepackage{booktabs}

\usepackage{amssymb, paralist, xspace, url, amscd, euscript, mathrsfs, stmaryrd}

\DeclareFontFamily{U}{rsfs}{%
	\skewchar\font127}
\DeclareFontShape{U}{rsfs}{m}{n}{%
	<-6>rsfs5<6-8.5>rsfs7<8.5->rsfs10}{}
\DeclareSymbolFont{rsfs}{U}{rsfs}{m}{n}
\DeclareSymbolFontAlphabet
{\mathrsfs}{rsfs}
\DeclareRobustCommand*\rsfs{%
	\@fontswitch\relax\mathrsfs}

\theoremstyle{plain}
\newtheorem{thm}{Theorem}[section]
\newtheorem{prop}[thm]{Proposition}
\newtheorem{lem}[thm]{Lemma}

\newtheorem{defi}[thm]{Definition}
\newtheorem{rmk}[thm]{Remark}
\newtheorem{cor}[thm]{Corollary}

\newtheorem{prop-defi}[thm]{Proposition-Definition}
\newtheorem{thm-defi}[thm]{Theorem-Definition}
\newtheorem{lem-defi}[thm]{Lemma-Definition}

\newtheorem{conj}[thm]{Conjecture}
\newtheorem{exam}[thm]{Example}

\newcommand{\cC}{\mathcal{C}}
\newcommand{\dD}{\mathcal{D}}
\newcommand{\eE}{\mathcal{E}}
\newcommand{\fF}{\mathcal{F}}

\newcommand{\hH}{\mathcal{H}}
\newcommand{\iI}{\mathcal{I}}

\newcommand{\kK}{\mathcal{K}}
\newcommand{\lL}{\mathcal{L}}
\newcommand{\mM}{\mathcal{M}}
\newcommand{\nN}{\mathcal{N}}
\newcommand{\oO}{\mathcal{O}}
\newcommand{\pP}{\mathcal{P}}

\newcommand{\rR}{\mathcal{R}}

\newcommand{\tT}{\mathcal{T}}
\newcommand{\uU}{\mathcal{U}}
\newcommand{\vV}{\mathcal{V}}
\newcommand{\wW}{\mathcal{W}}
\newcommand{\xX}{\mathcal{X}}
\newcommand{\yY}{\mathcal{Y}}
\newcommand{\zZ}{\mathcal{Z}}

\newcommand{\fM}{\mathfrak{M}}
\newcommand{\fN}{\mathfrak{N}}

\newcommand{\fU}{\mathfrak{U}}

\newcommand{\Supp}{\mathop{\rm Supp}\nolimits}
\newcommand{\Hom}{\mathop{\rm Hom}\nolimits}

\newcommand{\dR}{\mathbf{R}}

\newcommand{\id}{\textrm{id}}

\newcommand{\ch}{\mathop{\rm ch}\nolimits}

\newcommand{\Spec}{\mathop{\rm Spec}\nolimits}
\newcommand{\rank}{\mathop{\rm rank}\nolimits}
\newcommand{\Coh}{\mathop{\rm Coh}\nolimits}

\newcommand{\us}{\mathchar`-\rm{us}}

\newcommand{\cneq}{\mathrel{\raise.095ex\hbox{:}\mkern-4.2mu=}}
\newcommand{\eqcn}{\mathrel{=\mkern-4.5mu\raise.095ex\hbox{:}}}

\newcommand{\Aut}{\mathop{\rm Aut}\nolimits}

\newcommand{\dg}{\mathrm{dg}}

\newcommand{\RHom}{\mathop{\dR\mathrm{Hom}}\nolimits}
\newcommand{\Ker}{\mathop{\rm Ker}\nolimits}

\newcommand{\Ind}{\mathop{\rm Ind}\nolimits}
\newcommand{\MF}{\mathop{\rm MF}\nolimits}

\newcommand{\Crit}{\mathop{\rm Crit}\nolimits}

\newcommand{\coh}{\mathrm{coh}}
\newcommand{\qcoh}{\mathrm{qcoh}}
\newcommand{\Dbc}{D^b_{\rm{coh}}}

\newcommand{\inclusion}{\ar@<-0.3ex>@{^{(}->}[r]}
\newcommand{\upinclusion}{\ar@<-0.3ex>@{^{(}->}[u]}
\newcommand{\leinclusion}{\ar@<-0.3ex>@{^{(}->}[l]}
\newcommand{\doinclusion}{\ar@<-0.3ex>@{^{(}->}[d]}
\newcommand{\diasquare}{\ar@{}[rd]|\square}

\newcommand{\lkakko}{[\![}
\newcommand{\rkakko}{]\!]}

\newcommand{\C}{\mathbb{C}^{\ast}}
\newcommand{\rig}{\mathchar`-\rm{rig}}

\newcommand{\st}{\mathchar`-\rm{st}}

\newcommand{\dDT}{\mathcal{DT}}

\newcommand{\modu}{\mathchar`-\mathrm{mod}}

\makeatletter
\renewcommand{\theequation}{%
	\thesection.\arabic{equation}}
\@addtoreset{equation}{section}
\makeatother

\begin{document}
	
	\begin{abstract}
We prove two kinds of 
$\mathbb{Z}/2$-periodic
Koszul duality equivalences for
triangulated categories of matrix factorizations associated with 
$(-1)$-shifted cotangents over quasi-smooth affine derived schemes. 
We use this result to define $\mathbb{Z}/2$-periodic version of 
Donaldson-Thomas categories for local surfaces, whose $\C$-equivariant 
version was introduced and developed in the author's previous paper. 
We compare $\mathbb{Z}/2$-periodic DT category with the 
$\C$-equivariant one, and deduce wall-crossing equivalences of 
$\mathbb{Z}/2$-periodic DT categories from those of $\C$-equivariant
DT categories. 
	\end{abstract}
	
	\maketitle
	
	\setcounter{tocdepth}{1}
	\tableofcontents

	\section{Introduction}
	\subsection{Background}
	This paper is a complement of the author's previous paper~\cite{TocatDT}, where 
	we introduced and developed $\C$-equivariant categorical Donaldson-Thomas theory 
	for local surfaces, i.e. Calabi-Yau 3-folds obtained as 
	total spaces of canonical line bundles on 
	surfaces. 
	It was based on Koszul duality equivalence, formulated as follows. 
	Let $Y$ be a smooth affine scheme over $\mathbb{C}$ and $V \to Y$ a vector bundle with a section $s$. 
	Then the derived zero locus of $s$ is a quasi-smooth 
	affine derived scheme
	\begin{align*}
		\fU \cneq (s=0) \hookrightarrow Y. 
		\end{align*}
	 On the other hand, 
	the above data determines a function 
	on the total space of $V^{\vee} \to Y$
	\begin{align*}
		w \colon V^{\vee} \to \mathbb{C}, \ 
		w(x, v)=\langle s(x), v\rangle
		\end{align*}
	for $x \in Y$ and $v \in V^{\vee}|_{x}$. 
	The above function is of $\C$-weight two 
	with respect to the fiberwise weight two $\C$-action on $V^{\vee}$. 
	The Koszul duality equivalence proved in~\cite{MR3071664, MR2982435, MR3631231, TocatDT}
	is an equivalence of triangulated categories
	\begin{align}\label{intro:koszul}
		\Psi \colon \MF_{\coh}^{\C}(V^{\vee}, w) \stackrel{\sim}{\to}
		\Dbc(\fU). 
		\end{align}
	Here the left hand side is the triangulated
	category of $\C$-equivariant 
	matrix factorizations of $w$ (cf.~\cite{Orsin}), 
	and the right hand side is the derived category 
	of dg-modules over $\oO_{\fU}$ with bounded coherent 
	cohomologies. A key observation is that, under 
	the equivalence (\ref{intro:koszul}),  
	the supports of matrix factorizations correspond to 
	singular supports of coherent sheaves introduced in~\cite{MR3300415}. 
	The $\C$-equivariant DT category 
	is then defined in~\cite{TocatDT} as a certain singular 
	support quotient of the derived category  
	of coherent sheaves on a quasi-smooth derived stack, which 
	through the equivalence (\ref{intro:koszul}) is regarded 
	as a gluing of categories of 
	$\C$-equivariant matrix factorizations. 
	
	On the other hand, 
	triangulated categories of matrix factorizations 
	without $\C$-actions are in general $\mathbb{Z}/2$-periodic. 
	As there is no $\C$-action on a general 
	CY 3-fold, it is more suitable to define
	$\mathbb{Z}/2$-periodic version of 
	DT categories rather than $\C$-equivariant ones. 
	An obstruction toward the $\mathbb{Z}/2$-periodic version 
	is a failure of the equivalence (\ref{intro:koszul}), i.e. 
	the $\mathbb{Z}/2$-periodic triangulated 
	categories of matrix factorizations $\MF^{\mathbb{Z}/2}_{\coh}(V^{\vee}, w)$ is 
	not necessary equivalent to the 
	$\mathbb{Z}/2$-periodic derived categories of coherent sheaves on $\fU$. 
	
	In this paper, we observe that the above issue is caused by a subtlety 
	on a correct 
	definition of $\mathbb{Z}/2$-periodic derived category for $\fU$. 
	We show that there is a $\mathbb{Z}/2$-periodic version of the equivalence (\ref{intro:koszul})
	by using either Positselski's $\mathbb{Z}/2$-periodic derived categories of the second kind~\cite{Postwo}, 
	or singular support quotients of usual $\mathbb{Z}$-graded derived categories 
	of coherent sheaves. We will use the latter 
	$\mathbb{Z}/2$-periodic equivalence to define 
	$\mathbb{Z}/2$-periodic DT categories. 
	We will show that it is recovered from the $\C$-equivariant one 
	up to idempotent completion, and use 
	this result to deduce some wall-crossing equivalence 
	of $\mathbb{Z}/2$-periodic DT categories from the results of 
	$\C$-equivariant DT categories.

	\subsection{$\mathbb{Z}/2$-periodic Koszul duality}
	In this paper, we 
	show the $\mathbb{Z}/2$-periodic version of the equivalence (\ref{intro:koszul}). 
	It is stated as follows: 
	\begin{thm}\emph{(Theorem~\ref{thm:Z2kos})}\label{thm:intro:KZ2}
		There exist equivalences 
		\begin{align*}
			\MF_{\qcoh}^{\mathbb{Z}/2}(V^{\vee}, w)
			\stackrel{\sim}{\to}
			\Ind \Dbc(\fU_{\epsilon})/\Ind \cC_{\Crit(w) \times \{0\}}
			\stackrel{\sim}{\to}
			D^{\rm{co}}(\oO_{\fU}\modu^{\mathbb{Z}/2}),
			\end{align*}
			which restricts to the commutative diagram 
		\begin{align}\label{intro:cdiag}
			\xymatrix{
				\MF_{\coh}^{\mathbb{Z}/2}(V^{\vee}, w) \ar[d]_-{\Psi_{\epsilon}}^-{\sim} \ar@<-0.3ex>@{^{(}->}[r] & 
				\overline{\MF}_{\coh}^{\mathbb{Z}/2}(V^{\vee}, w)	\ar[d]^-{\Psi^{\mathbb{Z}/2}}_-{\sim} \\
				\Dbc(\fU_{\epsilon})/\cC_{\Crit(w) \times \{0\}} \ar@<-0.3ex>@{^{(}->}[r]^-{\Upsilon} & 
				\overline{D}^{\rm{abs}}(\oO_{\fU}\modu_{\rm{fg}}^{\mathbb{Z}/2}). 	
			}
		\end{align}
	Here $\overline{(-)}$ indicates the 
	idempotent completion and 
	the horizontal arrows are fully-faithful with dense images. 
		\end{thm}
	Here we explain the notation of the above theorem. 
	The derived scheme $\fU_{\epsilon}$ is defined by 
	$\fU_{\epsilon} \cneq \fU \times \Spec \mathbb{C}[\epsilon]$
	with $\deg(\epsilon)=-1$. 
	Its $(-1)$-shifted cotangent 
	has underlying classical scheme 
	given by $\Crit(w) \times \mathbb{A}^1$. 
	The subcategory 
	\begin{align*}
		\cC_{\Crit(w) \times \{0\}} \subset \Dbc(\fU_{\epsilon})
		\end{align*}
	consists of objects whose singular supports are contained in 
	$\Crit(w) \times \{0\} \subset \Crit(w) \times \mathbb{A}^1$. 
	The triangulated categories 
	$D^{\rm{co}}(\oO_{\fU} \modu^{\mathbb{Z}/2})$, 
	$D^{\rm{abs}}(\oO_{\fU}\modu_{\rm{fg}}^{\mathbb{Z}/2})$ are
	derived categories of $\mathbb{Z}/2$-periodic dg $\oO_{\fU}$-modules 
	of the second kind~\cite{Postwo}, which we 
	will review in Subsection~\ref{subsec:review:dcat}. 
	These are not necessary equivalent to the usual 
	$\mathbb{Z}/2$-periodic derived categories, 
	e.g. an object with trivial cohomologies may 
	be non-zero in these categories. 
	The former category $D^{\rm{co}}(\oO_{\fU} \modu^{\mathbb{Z}/2})$
	is regraded as a $\mathbb{Z}/2$-periodic analogue of 
	ind-coherent sheaves~\cite{MR3136100}
	(see Remark~\ref{rmk:indcoh}), though 
	it is not an ind-completion of the usual $\mathbb{Z}/2$-periodic 
	derived category of coherent sheaves. 
	
	\subsection{$\mathbb{Z}/2$-periodic categorical DT theory}
	We use the left vertical arrow in (\ref{intro:cdiag}) to define 
	$\mathbb{Z}/2$-periodic DT categories. 
	Let $\fM$ be a quasi-smooth and QCA derived stack (see Subsection~\ref{subsec:qsmooth})
	with classical truncation $\mM=t_0(\fM)$,  
	and $\Omega_{\fM}[-1]$ its $(-1)$-shifted cotangent derived stack. 
	For an open substack 
	\begin{align*}
		\nN^{\rm{ss}} \subset \nN \cneq t_0(\Omega_{\fM}[-1])
		\end{align*}
	with complement $\zZ$, we
	set 
	\begin{align*}
		\zZ_{\epsilon} \cneq \C(\zZ \times \{1\}) \sqcup 
		(\nN \times \{0\})
		\subset \nN \times \mathbb{A}^1. 
		\end{align*}
	Here $\C$ acts on the fibers of $\nN \to \mM$
	and $\mathbb{A}^1$ by weight two. 
	Then $\zZ_{\epsilon}$ is a conical closed substack 
	in the $(-1)$-shifted cotangent over 
	$\fM_{\epsilon} \cneq \fM \times \Spec \mathbb{C}[\epsilon]$
	for $\deg(\epsilon)=-1$. 
	We define the 
	$\mathbb{Z}/2$-periodic DT category 
	$\dDT^{\mathbb{Z}/2}(\nN^{\rm{ss}})$ by 
	\begin{align}\label{intro:defDT}
		\dDT^{\mathbb{Z}/2}(\nN^{\rm{ss}})
		\cneq \Dbc(\fM_{\epsilon})/\cC_{\zZ_{\epsilon}}. 
		\end{align}
	By the left vertical equivalence in (\ref{intro:cdiag}), 
	the above category may be regarded as a gluing of $\mathbb{Z}/2$-periodic 
	categories of matrix factorizations. 
	Its dg-enhancement $\dDT^{\mathbb{Z}/2}(\nN^{\rm{ss}})_{\rm{dg}}$
	is similarly defined by taking 
	the Drinfeld quotient instead of Verdier quotient. 
	We show that, up to idempotent completion, 
	the $\mathbb{Z}/2$-periodic DT category is
	recovered from the $\C$-equivariant DT category 
	$\dDT^{\C}(\nN^{\rm{ss}})$ defined in~\cite{TocatDT}. 
	\begin{thm}\emph{(Theorem~\ref{thm:compareDT})}\label{thm:intro:compare}
		Suppose that $\zZ \subset \nN$ is $\C$-invariant 
		and $\Ind \cC_{\zZ} \subset \Ind \Dbc(\fM)$
		is compactly generated. 
		Then there is an equivalence 
			\begin{align}\notag
			\overline{\mathcal{DT}}^{\mathbb{Z}/2}(\nN^{\rm{ss}})_{\rm{dg}}
			\simeq \dR \underline{\hH om}(\mathbb{C}[u^{\pm 1}], 
			\Ind \mathcal{DT}^{\C}(\nN^{\rm{ss}})_{\rm{dg}})^{\rm{cp}}. 	
		\end{align}
	Here $\deg(u)=2$, $\mathbb{C}[u^{\pm 1}]$ is regarded as a 
	dg-category with one object, 
	and $\dR \underline{\hH om}(-, -)$
	is an inner Hom of dg-categories~\cite{Todg}. 
		\end{thm}
	
	\subsection{Wall-crossing equivalences of $\mathbb{Z}/2$-periodic DT categories for local surfaces}
	In~\cite{TocatDT}, we proved several wall-crossing equivalences 
	or fully-faithful functors
	of $\C$-equivariant DT categories for local surfaces. 
	We can apply these results to show 
	wall-crossing equivalences or fully-faithful 
	functors of $\mathbb{Z}/2$-periodic DT categories 
	using Theorem~\ref{thm:intro:compare}. 
	We give one of such examples: wall-crossing equivalence of 
	DT categories for 
	one dimensional stable sheaves. 
	
	Let $S$ be a smooth projective surface and 
	\begin{align*}
		\pi \colon X=\mathrm{Tot}_S(\omega_S) \to S
		\end{align*}
	 the 
	associated local surface. 
	For a stability condition $\sigma$
	and a numerical class $v$, 
	we denote by $M_X^{\sigma}(v)$ the moduli space of 
	$\sigma$-semistable compactly supported coherent sheaves on $X$. 
	Under the assumption that there is no strictly $\sigma$-semistable sheaves, 
	we define the $\mathbb{Z}/2$-periodic DT category 
	\begin{align}\notag
		\dDT^{\mathbb{Z}/2}(M_X^{\sigma}(v))
		\end{align}
	for $M_X^{\sigma}(v)$, 
	following the construction of the basic model (\ref{intro:defDT})
	(see Definition~\ref{def:period}). 
	
	Let $v$ be a primitive numerical class of a one dimensional 
	sheaf, and suppose that $\sigma$ lies in a chamber of 
	the space of stability conditions. 
	We have the following wall-crossing diagram 
	\begin{align*}
		\xymatrix{
			M_X^{\sigma_+}(v) \ar[rd] & & \ar[ld] M_X^{\sigma_-}(v) \\
			& M_X^{\sigma}(v) &
		}
	\end{align*}
which 
	is a \textit{d-critical flop}
	in the sense of~\cite{Toddbir}, i.e. it is a d-critical analogue 
	of flop in birational geometry. 
As an analogy of D/K equivalence conjecture~\cite{B-O2, MR1949787},
we conjecture that 
the DT categories are equivalent under the above wall-crossing diagram. 
In~\cite{TocatDT}, we proved the wall-crossing equivalence for 
$\C$-equivariant DT categories under some assumption of 
preservation of stability under push-forward to the surface. 
By directly applying Theorem~\ref{thm:intro:compare}, we have the following: 
	\begin{thm}\label{intro:thm:period2}
		Suppose that any $\sigma$-semistable sheaf on $X$
		with numerical class $v$ push-forwards to a $\sigma$-semistable 
		sheaf on $S$. 
	Then we have an equivalence
	\begin{align}\notag
		\overline{\dD\tT}^{\mathbb{Z}/2}(M_{X}^{\sigma_+}(v)) \stackrel{\sim}{\to}
		\overline{\dD\tT}^{\mathbb{Z}/2}(M_{X}^{\sigma_-}(v)).
	\end{align}
\end{thm}
The examples where the assumption of Theorem~\ref{intro:thm:period2}
is satisfied
are discussed in~\cite[Section~5.4.4]{TocatDT}, e.g. 
it is satisfied when the curve class of $v$ is reduced. 
In a similar way, we can also prove $\mathbb{Z}/2$-periodic 
version of wall-crossing equivalences or fully-faithful 
functors of DT categories for MNOP/PT moduli spaces, proved 
in~\cite{TocatDT} for $\C$-equivariant case (see~Remark~\ref{rmk:DK}). 
	
	\subsection{Acknowledgements}
	The author would like to thank Yalong Cao, Hsueh-Yung Lin, Tudor P{\u{a}}durairu and 
	Wai-kit Yeung for helpful discussions. 
	The author is supported by World Premier International Research Center
	Initiative (WPI initiative), MEXT, Japan, and Grant-in Aid for Scientific
	Research grant (No.~19H01779) from MEXT, Japan.

	\subsection{Notation and convention}\label{subsec:notation}
	In this paper, all the schemes or (derived) stacks
	are locally of finite presentation over 
	$\mathbb{C}$. 
	For a scheme or derived stack $Y$
	and a quasi-coherent sheaf $\fF$ on it, we denote by 
	$S_{\oO_Y}(\fF)$ its symmetric product
	$\oplus_{i\ge 0} \mathrm{Sym}^i_{\oO_Y}(\fF)$. 
	We omit the subscript $\oO_Y$ if it is clear from the context. 
	For a derived stack $\mathfrak{M}$, we always 
	denote by $t_0(\mathfrak{M})$ the underived
	stack given by the truncation. 
	For an algebraic group $G$ which acts on $Y$, 
	we denote by $[Y/G]$ the associated 
	quotient stack. 
	
	In this paper, all the dg-categories or triangulated categories 
	are defined over 
	$\mathbb{C}$. 
	The category of dg-categories is denoted by 
	$\mathrm{dgCat}$. 
	Its objects consist of dg-categories over $\mathbb{C}$
	with morphisms given by dg-functors. 
	By Tabuada~\cite{Tab}, there is a cofibrantly generated 
	model category structure on 
	$\mathrm{dgCat}$, 
	whose localization by weak equivalences
	is denoted by
	$\mathrm{Ho}(\mathrm{dgCat})$. 
An equivalence between dg-categories
is defined to be 
 an isomorphism 
	in $\mathrm{Ho}(\mathrm{dgCat})$. 
		For
	a triangulated category $\dD$ and its 
	triangulated subcategory $\dD' \subset \dD$, 
	we denote by 
	$\dD/\dD'$
	its Verdier quotient. 
	In the case that $\dD$ is a dg-category 
	and $\dD' \subset \dD$ is a dg-subcategory, 
	its Drinfeld dg-quotient~\cite{MR3037900} 
	is also denoted by $\dD/\dD'$. 
	We denote by $\dD^{\rm{cp}} \subset \dD$ the 
	subcategory of compact objects. 
	A subcategory $\dD' \subset \dD$ is called dense 
	if any object in $\dD$ is a direct summand of an object in $\dD'$. 
	
For a dg-category $\dD$, 
we denote by $\Ind \dD$ 
its dg-categorical 
	ind-completion of $\dD$ (denoted as $\widehat{\dD}$
	in~\cite[Section~7]{Todg}, also see~\cite[Section~5.3.5]{Ltopos} in the 
	context of $\infty$-category). 
For dg-categories $\dD_1, \dD_2$, we denote by 
$\dD_1 \otimes \dD_2$ the dg-category 
whose set of objects is $\mathrm{Obj}(\dD_1) \times \mathrm{Obj}(\dD_2)$
and 
\begin{align*}
	\Hom^{\ast}_{\dD_1 \otimes \dD_2}((E_1, E_2), (F_1, F_2))
	=\Hom^{\ast}_{\dD_1}(E_1, F_1) \otimes 
	\Hom^{\ast}_{\dD_2}(E_2, F_2). 
	\end{align*}
By~\cite[Corollary~6.4]{Todg}, 
for two dg-categories $\dD_1, \dD_2$, there exists 
an inner Hom of dg-category 
$\dR \underline{\hH om}(\dD_1, \dD_2) \in \mathrm{Ho}(\mathrm{dgCat})$
satisfying that (see~\cite[Corollary~7.6]{Todg})
\begin{align}\notag
	\dR \underline{\hH om}(\dD_1, \Ind{\dD}_2) 
	\simeq 
	\Ind(\dD_1^{\rm{op}} \otimes \dD_2). 
	\end{align}
	When we discuss limits, ind-completions, 
inner Hom for triangulated categories, 
	we (implicitly or explicitly) take these functors on dg-enhancements
		or $\infty$-categorical enhancements 
(which will be obviously given in the context)
	and then take their homotopy categories. 
	
	\section{Some background}
	In this section, we review several background used in this paper: 
	derived categories of first and second kind in~\cite{Postwo}, 
	dg or triangulated categories of factorizations~\cite{Ornonaff, MR3366002, MR3112502}, $\C$-equivariant 
	Koszul duality~\cite{MR3071664, MR2982435, MR3631231, TocatDT} and singular supports of (ind) coherent sheaves~\cite{MR3300415}. 
	\subsection{Review of derived categories}\label{subsec:review:dcat}
	Let $R$ be a commutative differential graded algebra over $\mathbb{C}$ with 
	non-positive degrees. 
	Let $\Gamma$ be either $\mathbb{Z}$ or $\mathbb{Z}/2$, 
	and we regard
	$R$ as a $\Gamma$-graded dg-algebra. 
	We recall some basic terminology of 
	derived categories of $\Gamma$-graded dg-modules over $R$, 
	following~\cite{Postwo}. 
	
	We denote by $R\modu^{\Gamma}$ the dg-category of $\Gamma$-graded 
	dg-modules over $R$, and $\mathrm{Ho}(R\modu^{\Gamma})$ its homotopy category. 
	We have the full dg or triangulated subcategories
	\begin{align*}
			\mathrm{Acy}_{\rm{dg}} \subset R\modu^{\Gamma}, \ 
		\mathrm{Acy} \subset \mathrm{Ho}(R\modu^{\Gamma})
		\end{align*}
	consisting of acyclic objects, i.e. 
	$\hH^i(M)=0$ for all $i \in \Gamma$. 
	The dg or triangulated derived categories of dg $R$-modules are defined 
	by the quotient categories
	\begin{align}\notag
		D(R\modu^{\Gamma})_{\rm{dg}} \cneq 
		R\modu^{\Gamma}/\mathrm{Acy}_{\rm{dg}}, \ 
		D(R\modu^{\Gamma}) \cneq \mathrm{Ho}(R\modu^{\Gamma})/\mathrm{Acy}. 
		\end{align}
	Here we take the Drinfeld quotient~\cite{MR2028075} for the quotient of 
	dg-categories, and take the Verdier quotient for the quotient of 
	triangulated categories. 
	The former dg-category is a dg-enhancement of the latter triangulated 
	category, i.e. 
	the homotopy category of $D(R\modu^{\Gamma})_{\rm{dg}}$ is 
	equivalent to $D(R\modu^{\Gamma})$. 
	We also have the full dg or triangulated subcategories
	\begin{align}\notag
		D_{\rm{fg}}(R\modu^{\Gamma})_{\rm{dg}} \subset
		D(R\modu^{\Gamma})_{\rm{dg}}, \ 
		D_{\rm{fg}}(R\modu^{\Gamma}) \subset D(R\modu^{\Gamma})
		\end{align}
	consisting of objects $M$ such that 
	$\oplus_{i \in \Gamma}\hH^i(M)$ is 
	finitely generated over the commutative algebra 
	$\hH^0(R)$. 
	The above dg or triangulated categories
	for $\Gamma=\mathbb{Z}$ are equivalent to (quasi) coherent sheaves on the 
	affine derived scheme $\Spec R$, so they are `usual' derived 
	categories. 
		
	There are other kinds of constructions of derived categories
	by Positselski~\cite{Postwo}. 
		We have the full dg or triangulated subcategories
	\begin{align*}
		\mathrm{Acy}^{\rm{co}}_{\rm{dg}} \subset 
		R\modu^{\Gamma}, \ 
		\mathrm{Acy}^{\rm{co}} \subset 
		\mathrm{Ho}(R\modu^{\Gamma})
	\end{align*}
where $\mathrm{Acy}^{\rm{co}}$
is the smallest thick triangulated subcategory 
which contains totalizations of short exact sequences of 
dg $R$-modules and closed under taking infinite direct sums, 
and $\mathrm{Acy}^{\rm{co}}_{\rm{dg}}$ consists of objects 
which are isomorphic to objects in $\mathrm{Acy}^{\rm{co}}$
in the homotopy category. 
An object in the above subcategories is called \textit{coacyclic}. 
The dg or triangulated \textit{coderived categories}
are defined by 
\begin{align*}
		D^{\rm{co}}(R\modu^{\Gamma})_{\rm{dg}}\cneq 
	R\modu^{\Gamma}/\mathrm{Acy}^{\rm{co}}_{\rm{dg}}, \ 
		D^{\rm{co}}(R\modu^{\Gamma})\cneq 
	\mathrm{Ho}(R\modu^{\Gamma})/\mathrm{Acy}^{\rm{co}}. 
	\end{align*}

		Let $R\modu_{\rm{fg}}^{\Gamma}$
	be the dg subcategory of $R\modu^{\Gamma}$ consisting of finitely 
	generated dg $R$-modules, and $\mathrm{Ho}(R\modu_{\rm{fg}}^{\Gamma})$ its 
	homotopy category. 
		We have the full dg or triangulated subcategories
	\begin{align*}
		\mathrm{Acy}^{\rm{abs}}_{\rm{dg}} \subset 
		R\modu_{\rm{fg}}^{\Gamma}, \ 
		\mathrm{Acy}^{\rm{abs}} \subset 
		\mathrm{Ho}(R\modu_{\rm{fg}}^{\Gamma})
	\end{align*}
	where $\mathrm{Acy}^{\rm{abs}}$
	is the smallest thick triangulated subcategory 
	which contains totalizations of short exact sequences of finitely generated
	dg $R$-modules, 
	and $\mathrm{Acy}^{\rm{abs}}_{\rm{dg}}$ consists of objects 
	which are isomorphic to objects in $\mathrm{Acy}^{\rm{abs}}$
	in the homotopy category. An object in the above subcategories is called \textit{absolutely acyclic}. 
	The dg or triangulated \textit{absolute derived categories}
	are defined by 
	\begin{align*}
		D^{\rm{abs}}(R\modu_{\rm{fg}}^{\Gamma})_{\rm{dg}}\cneq 
		R\modu_{\rm{fg}}^{\Gamma}/\mathrm{Acy}^{\rm{abs}}_{\rm{dg}}, \ 
		D^{\rm{abs}}(R\modu_{\rm{fg}}^{\Gamma})\cneq 
		\mathrm{Ho}(R\modu_{\rm{fg}}^{\Gamma})/\mathrm{Acy}^{\rm{abs}}. 
	\end{align*}
	
The inclusion $R\modu^{\Gamma}_{\rm{fg}} \subset R\modu^{\Gamma}$
induces the functor 
\begin{align}\label{abs:co}
	D^{\rm{abs}}(R\modu_{\rm{fg}}^{\Gamma})
\to	D^{\rm{co}}(R\modu^{\Gamma}) 
	\end{align}
which is fully-faithful 
by~\cite[Section~11, Theorem~1]{Postwo}. 
Also since any coacyclic object is 
acyclic, 
there is a natural functor
\begin{align}\label{funct:coabs}
	D^{\rm{co}}(R\modu^{\Gamma})
	\to D(R\modu^{\Gamma}). 
	\end{align}
In the $\mathbb{Z}$-graded case, we have 
the following lemma which essentially follows from~\cite[Section~3.4, Theorem~1]{Postwo}. 
	\begin{lem}\label{lem:equivDb}
		Suppose that
		$\Gamma=\mathbb{Z}$, 
		 $R^i=0$ for $i\ll 0$ and each $\hH^i(R)$ is finitely 
		 generated over $\hH^0(R)$. 
		Then the functor (\ref{funct:coabs}) restricts to the equivalence
		\begin{align*}
			D^{\rm{abs}}(R\modu_{\rm{fg}}^{\mathbb{Z}}) \stackrel{\sim}{\to}
			D_{\rm{fg}}(R\modu^{\mathbb{Z}}). 
			\end{align*}
		\end{lem}
	\begin{proof}
		Let $\mathrm{Ho}^+(R\modu^{\mathbb{Z}})$
		be the homotopy category of $\mathbb{Z}$-graded 
		dg $R$-modules $M$
		such that $M^i=0$ for $i\ll 0$. 
		We set 
		\begin{align*}
			\mathrm{Acy}^+ \cneq \mathrm{Acy} \cap 
			\mathrm{Ho}^+(R\modu^{\mathbb{Z}}), \ 
			\mathrm{Acy}^{\rm{co}, +} \cneq \mathrm{Acy}^{\rm{co}} \cap 
			\mathrm{Ho}^+(R\modu^{\mathbb{Z}}). 
			\end{align*}
		Then we have natural fully-faithful functors 
		by~\cite[Section~3.4, Theorem~1 (b), (c)]{Postwo}
		\begin{align}\label{ff:Pos1}
			\mathrm{Ho}^+(R\modu^{\mathbb{Z}})/\mathrm{Acy}^+
			\hookrightarrow D(R\modu^{\mathbb{Z}}), \ 
			\mathrm{Ho}^+(R\modu^{\mathbb{Z}})/\mathrm{Acy}^{\rm{co}, +}
			\hookrightarrow D^{\rm{co}}(R\modu^{\mathbb{Z}}). 
			\end{align}
	By the assumption on $R$, 
	the functor (\ref{abs:co}) factors through the functors
				\begin{align}\label{ff:Pos2}
			D^{\rm{abs}}(R\modu_{\rm{fg}}^{\mathbb{Z}}) \to 
				\mathrm{Ho}^+(R\modu^{\mathbb{Z}})/\mathrm{Acy}^{\rm{co}, +}
				\to D^{\rm{co}}(R\modu^{\mathbb{Z}}). 			
			\end{align}
		The above composition is fully-faithful 
		and the second functor is 
		also fully-faithful since the second 
		functor in (\ref{ff:Pos1})
		is fully-faithful. 
		Therefore the first functor in (\ref{ff:Pos2}) is also fully-faithful. 
		On the other hand, by~\cite[Section~3.4, Theorem~1 (a)]{Postwo}
		we have $\mathrm{Acy}^+=\mathrm{Acy}^{\rm{co}, +}$. 
		Therefore combined with the first 
		functor in (\ref{ff:Pos1}), we obtain the fully-faithful functors
		\begin{align*}
				D^{\rm{abs}}(R\modu_{\rm{fg}}^{\mathbb{Z}}) \hookrightarrow
				\mathrm{Ho}^+(R\modu^{\mathbb{Z}})/\mathrm{Acy}^{\rm{co}, +}
				=
			\mathrm{Ho}^+(R\modu^{\mathbb{Z}})/\mathrm{Acy}^{+}
			\hookrightarrow D(R\modu^{\mathbb{Z}}). 
			\end{align*}
		Any object in the image of the above composition has 
		bounded finitely generated cohomologies by the assumption on $R$, 
		so we obtain the fully-faithful functor 
		\begin{align}\label{ff:Pos3}
			D^{\rm{abs}}(R\modu^{\mathbb{Z}}_{\rm{fg}}) \hookrightarrow 
			D_{\rm{fg}}(R\modu^{\mathbb{Z}}). 
			\end{align}
		It remains to show that the above functor is essentially 
		surjective. 
		Note that any object in the target is given by a finite 
		successive extensions of its cohomologies. 
		Since each cohomology is finitely generated over $\hH^0(R)$, 
		it comes from the 
		objects in the source of (\ref{ff:Pos3}).
		Since (\ref{ff:Pos3}) is fully-faithful, we conclude that
		it is also essentially surjective.  
		\end{proof}
	
	\begin{rmk}\label{rmk:Z2}
		As we will see in Example~\ref{exam:3}, 
		the result of Lemma~\ref{lem:equivDb}
		does not hold in the $\mathbb{Z}/2$-graded case. 		
		\end{rmk}
	\begin{rmk}\label{rmk:indcoh}
		By~\cite[Section~3.11, Theorem~1]{Postwo},
		the coderived category 
		$D^{\rm{co}}(R\modu^{\Gamma})$ 
		is compactly generated by 
		$D^{\rm{abs}}(R\modu^{\Gamma})$.  
		In the situation of Lemma~\ref{lem:equivDb}, 
		it follows that $D^{\rm{co}}(R\modu^{\mathbb{Z}})$
		is equivalent to the category of ind-coherent sheaves on 
		$\Spec R$ studied in~\cite{MR3136100} (see~\cite[Section~H.3]{MR3300415}), 
		and the functor (\ref{funct:coabs}) is a natural functor 
		from ind-coherent sheaves to quasi-coherent sheaves. 
		\end{rmk}

	\subsection{Review of factorization categories}\label{subsec:fact}
	Here we review the theory of factorizations 
associated with super-potentials. The basic references are~\cite{Ornonaff, MR3366002, MR3112502}. 

Let $\xX$ be a noetherian smooth algebraic stack over $\mathbb{C}$, 
$\lL \to \xX$ a line bundle and 
$w \in \Gamma(\xX, \lL^{\otimes 2})$ a global section. 
A \textit{(quasi) coherent factorization} of $w$ consists of 
\begin{align*}
	(\pP, d_{\pP}), \ d_{\pP} \colon \pP \to \pP \otimes \lL
	\end{align*}
where $\pP$ is a (quasi) coherent sheaf on $\xX$ and 
$d_{\pP}$ is a morphism of (quasi) coherent sheaves
satisfying $d_{\pP}^2=w$. 
The category of factorizations of $w$ naturally forms
a dg-category
denoted by $\underline{\MF}_{\star}(\xX, w)_{\rm{dg}}$ for
$\star \in \{\qcoh, \coh\}$, whose 
homotopy category 
$\mathrm{HMF}_{\star}(\xX, w)$
is a triangulated category. 
Let 
\begin{align}\notag
	\mathrm{Acy}_{\star}^{\rm{abs}} \subset \mathrm{HMF}_{\star}(\xX, w)
	\end{align}
be the minimal thick triangulated subcategory 
which contains totalizations of short exact sequences of 
(quasi) coherent factorizations of $w$. 
An object in $\mathrm{Acy}_{\star}^{\rm{abs}}$ is called 
\textit{absolutely acyclic}. 
The triangulated category of factorizations of $w$ is defined by 
the Verdier quotient 
\begin{align*}
	\MF_{\star}(\xX, w) \cneq 
	\mathrm{HMF}_{\star}(\xX, w)/\mathrm{Acy}^{\rm{abs}}_{\star}, \ 
	\star \in \{\qcoh, \coh\}. 
	\end{align*}
It admits a natural dg-enhancement by taking the Drinfeld quotient 
\begin{align}\label{MF:dg}
	\MF_{\star}(\xX, w)_{\rm{dg}} \cneq 
	\underline{\MF}_{\star}(\xX, w)_{\rm{dg}}/\mathrm{Acy}^{\rm{abs}}_{\star, \rm{dg}}, \
	\star \in \{\qcoh, \coh\}
	\end{align}
where $\mathrm{Acy}^{\rm{abs}}_{\star, \rm{dg}}$
is the full dg-subcategory 
of $\underline{\MF}_{\star}(\xX, w)_{\rm{dg}}$ consisting of 
absolutely acyclic objects in the homotopy category. 

Let $j \colon \uU \subset \xX$ be an open substack 
with complement $\zZ \subset \xX$. 
We define 
\begin{align}\label{MF:supp}
	\MF_{\star}(\xX, w)_{\zZ} \cneq 
	\Ker(\MF_{\star}(\xX, w) \stackrel{j^{\ast}}{\to} \MF_{\star}(\uU, w|_{\uU})). 
	\end{align}
Then we have the equivalence (cf.~\cite[Theorem~1.10]{MR3366002})
\begin{align}\label{equiv:supporZ}
	\MF_{\star}(\xX, w)/\MF_{\star}(\xX, w)_{\zZ}
	\stackrel{\sim}{\to} \MF_{\star}(\uU, w|_{\uU}). 
	\end{align}
Let $\Crit(w) \subset \xX$ be the critical locus. We also have the equivalence
(cf.~\cite[Corollary~5.3]{MR3112502})
	\begin{align}\label{Crit:supp}
		\MF_{\star}(\xX, w)_{\Crit(w)} 
		\stackrel{\sim}{\to} \MF_{\star}(\xX, w). 
		\end{align}
	
	We will use the following special construction of factorizations. 
	Let $\vV \to \xX$ be a vector bundle with 
	sections 
	$s \in \Gamma(\xX, \vV\otimes \lL)$, 
	$t \in \Gamma(\xX, \vV^{\vee} \otimes \lL)$
	satisfying 
	$w=\langle s, t \rangle$. 
	Then we have the following factorization of 
	$w$ (called \textit{Koszul factorization})
	\begin{align}\label{Kfact}
		\kK_{s, t} \cneq \left(\bigwedge^{\bullet}\vV^{\vee}, d_{\kK}\right), \ 
		d_{\kK} \colon  \bigwedge^i \vV^{\vee} \stackrel{(s, t)}{\to} \left( \bigwedge^{i-1}\vV^{\vee} \oplus 
		\bigwedge^{i+1}\vV^{\vee}  \right)\otimes \lL. 
		\end{align}
	If $s$ is a regular section 
	so that $\zZ \cneq (s=0) \subset \xX$ has codimension 
	$\rank \vV$, 
	then 
	we have 
	an isomorphism
	in $\mathrm{MF}_{\coh}(\xX, w)$ (see~\cite[Proposition~3.20]{MR3270588})
	\begin{align}\label{Kos:isom}
		\kK_{s, t} \cong (\oO_{\zZ}, d_{\oO_{\zZ}}=0). 
		\end{align}

	We will use the following two special versions of factorization categories. 
	Let $\yY$ be a smooth stack and set $\xX=\yY \times B\mu_2$, 
	and $\lL$ to be the line bundle on $\xX$ induced by the 
	weight one $\mu_2$-character. 
	Then $\lL^{\otimes 2} \cong \oO_{\xX}$, so any 
	regular function $w \colon \yY \to \mathbb{C}$
	determines a global section $w \in \Gamma(\xX, \lL^{\otimes 2})$. 
	We set
	\begin{align*}
		\MF_{\star}^{\mathbb{Z}/2}(\yY, w) \cneq 
		\MF_{\star}(\yY \times B\mu_2, w), \ 
		\star \in \{\qcoh, \coh\}. 
		\end{align*}
	The above triangulated category is $\mathbb{Z}/2$-periodic, i.e. 
	$[2] \cong\id$. 
	When $\yY$ is an affine scheme, the 
	above triangulated category is equivalent to 
	Orlov's triangulated category of matrix factorizations~\cite{Orsin}. 
	
	Let $\C$ acts on a smooth stack $\yY$ and set 
	$\xX=[\yY/\C]$, and $\lL$ to be the line bundle on $\xX$
	induced by the weight one $\C$-character. 
	For a global section $w \in \Gamma(\xX, \lL^{\otimes 2})$, 
	we set
	\begin{align*}
		\MF_{\star}^{\C}(\yY, w) \cneq \MF_{\star}([\yY/\C], w), \ 
		\star \in \{\qcoh, \coh\}. 
		\end{align*}
	For example, 
	we will use the above 
	construction for $\yY=[Y/G]$ for a
	noetherian scheme $Y$ with an action of an algebraic group $G$
	and an action of $\C$ which commutes with the $G$-action, 
	and $w \colon Y \to \mathbb{C}$ is a $G$-invariant 
	function with $\C$-weight two. 
	
	For a closed substack $\zZ \subset \yY$, 
	we define the subcategories 
	\begin{align*}
	\MF_{\star}^{\mathbb{Z}/2}(\yY, w)_{\zZ} \subset \MF_{\star}^{\mathbb{Z}/2}(\yY, w), \ 
		\MF_{\star}^{\C}(\yY, w)_{\zZ} \subset \MF_{\star}^{\C}(\yY, w)
		\end{align*}
	in the similar way as (\ref{MF:supp}). Here we assume that $\zZ$ is $\C$-invariant in the latter case. 
	The dg-categories 
	$\MF^{\ast}_{\star}(\yY, w)_{\rm{dg}}$, 
	$\MF^{\ast}_{\star}(\yY, w)_{\zZ, \rm{dg}}$ are 
	also defined in the similar way as (\ref{MF:dg}).

	\subsection{$\C$-equivariant Koszul duality}
	Let $Y$ be a smooth affine $\mathbb{C}$-scheme
	and $V \to Y$ a vector bundle on it. 
	Given a section $s \colon Y \to V$ of $V$, 
	its derived zero locus $\fU$ is given by 
	\begin{align}\label{frak:U}
		\fU =\Spec \rR(V \to Y, s)
		\end{align}
	where $\rR(V \to Y, s)$ is the Koszul complex 
	\begin{align*}
	\rR(V \to Y, s) \cneq 
	\left(  
	\cdots \to \bigwedge^2 V^{\vee} \stackrel{s}{\to} V^{\vee} \stackrel{s}{\to}
	\oO_Y \to 0 \right).	
		\end{align*}
	Its dg or triangulated 
	derived categories of coherent sheaves are defined by 
	\begin{align*}
		\Dbc(\fU)_{\rm{dg}} \cneq D_{\rm{fg}}(\oO_{\fU}\modu^{\mathbb{Z}})_{\rm{dg}}, \ 
		\Dbc(\fU) \cneq D_{\rm{fg}}(\oO_{\fU}\modu^{\mathbb{Z}}). 
		\end{align*}
	The ind-completion
	$\Ind \Dbc(\fU)$ of $\Dbc(\fU)$
	is defined to be the homotopy category of 
	$\Ind \Dbc(\fU)_{\rm{dg}}$ (see~Subsection~\ref{subsec:notation}), which 
	is equivalent to 
	$D^{\rm{co}}(\oO_{\fU}\modu^{\mathbb{Z}})$
	(see~Remark~\ref{rmk:indcoh}). 
	
		Let $V^{\vee} \to Y$ be the total space of the dual vector bundle of $V$. 
	There is an associated function on $V^{\vee}$, given by 
	\begin{align}\label{func:w}
		w \colon V^{\vee} \to \mathbb{C}, \ 
		w(x, v)=\langle s(x), v \rangle, \ 
		x \in Y, \ v \in V^{\vee}|_{x}. 
		\end{align}
	It is well-known that the critical locus of the above 
	function is the classical truncation of the $(-1)$-shifted cotangent 
	over $\fU$ (see~\cite{MR3607000}), 
	\begin{align*}
		t_0(\Omega_{\fU}[-1])=\Crit(w) \subset V^{\vee}. 
		\end{align*}
		Let $\C$ acts on the fibers of $V^{\vee} \to Y$ by weight two, so that $w$ is 
		of weight two. 
	We have the following Koszul duality equivalence 
	which relates derived category coherent sheaves on $\fU$ with
	the triangulated category of 
	$\C$-equivariant factorizations of $w$
	(see~\cite[Theorem~2.3.3, Lemma~2.3.10]{TocatDT}):
	
	\begin{thm}\emph{(cf.~\cite{MR3071664, MR2982435, MR3631231, TocatDT})}
		\label{thm:knoer}
		There is an equivalence of triangulated categories
		\begin{align}\label{equiv:Psi}
			\Psi \colon \MF_{\coh}^{\C}(V^{\vee}, w)
			\stackrel{\sim}{\to} \Dbc(\fU),
			\end{align}
		which extends to the equivalence 
		\begin{align}\label{ind:Psi}
			\Psi \colon \MF_{\qcoh}^{\C}(V^{\vee}, w) \stackrel{\sim}{\to}
			\Ind \Dbc(\fU). 
			\end{align}
	\end{thm}
The equivalence (\ref{equiv:Psi}) is constructed in the following way. 
Let $\kK_s$ be the following $\C$-equivariant
factorization of $w$
\begin{align}\notag
	\kK_s \cneq 
	\left(
	\oO_{V^{\vee}} \otimes_{\oO_Y} \oO_{\mathfrak{U}}, 
	d_{\kK_s} \right). 
\end{align}
Here 
the $\C$-action is given by the grading  
\begin{align*}
	\oO_{V^{\vee}} \otimes_{\oO_Y}\oO_{\mathfrak{U}}=
	S_{\oO_Y}(V[-2]) \otimes_{\oO_Y} S_{\oO_Y}(V^{\vee}[1]),
\end{align*}
and the weight one map $d_{\kK_s}$ is given by 
\begin{align}\label{diff:ds}
	d_{\kK_s}=
	1 \otimes d_{\oO_{\mathfrak{U}}}+\eta \colon \oO_{V^{\vee}} \otimes_{\oO_Y}\oO_{\mathfrak{U}} \to
	\oO_{V^{\vee}} \otimes_{\oO_Y}\oO_{\mathfrak{U}} \langle 1 \rangle,
\end{align}
where 
$\eta \in V\otimes_{\oO_Y} V^{\vee} \subset \oO_{V^{\vee}} \otimes_{\oO_Y}
\oO_{\fU}$
corresponds to 
$\id \in \Hom(V, V)$, and $\langle 1 \rangle$ indicates the shift of $\C$-weight by one. 
Then one can check that there is a quasi-isomorphism of dg-algebras
\begin{align*}
	\oO_{\fU} \stackrel{\sim}{\to}
	\Hom_{\MF_{\coh}^{\C}(V^{\vee}, w)_{\rm{dg}}}(\kK_s, \kK_s), 
	\end{align*}
and the equivalence $\Psi$ is given by (see the proof of~\cite[Theorem~2.3.3]{TocatDT})
\begin{align}\label{equiv:Psi:K}
	\Psi(-)=\RHom(\kK_s, -), \ 
	\MF_{\coh}^{\C}(V^{\vee}, w)
	\to \Dbc(\fU). 
	\end{align}
Its quasi-inverse is given by 
\begin{align}\label{equiv:Phi}
		\Phi(-)=\kK_s \otimes_{\oO_{\fU}} (-), \ 
	\Dbc(\fU) \to \MF_{\coh}^{\C}(V^{\vee}, w). 
	\end{align}

\subsection{Singular supports of (ind) coherent sheaves}
The theory of singular supports of 
coherent sheaves on $\mathfrak{U}$
is developed in~\cite{MR3300415}
following the earlier work~\cite{MR2489634}. 
Here we recall its definition. 
Let $\mathrm{HH}^{\ast}(\mathfrak{U})$ be the 
Hochschild cohomology
\begin{align*}
	\mathrm{HH}^{\ast}(\mathfrak{U})
	\cneq \Hom_{\mathfrak{U} \times \mathfrak{U}}^{\ast}
	(\Delta_{\ast}\oO_{\mathfrak{U}}, \Delta_{\ast}\oO_{\mathfrak{U}}).
\end{align*}
Here $\Delta \colon \mathfrak{U} \to \mathfrak{U} \times \mathfrak{U}$ is the diagonal. 
Then it is shown in~\cite[Section~4]{MR3300415}
that there 
exists a
canonical map 
$\hH^1(\mathbb{T}_{\mathfrak{U}}) \to \mathrm{HH}^2(\mathfrak{U})$, 
so the map of 
graded rings
\begin{align}\notag
	\oO_{\Crit(w)}=
	S(\hH^1(\mathbb{T}_{\mathfrak{U}})) \to \mathrm{HH}^{2\ast}(\mathfrak{U}) 
	\to \mathrm{Nat}_{D^b_{\rm{coh}}(\mathfrak{U})}(\id, \id[2\ast]). 
\end{align}
Here $\mathrm{Nat}_{D^b_{\rm{coh}}(\mathfrak{U})}(\id, \id[2\ast])$
is the group of natural transformations from $\id$ to $\id[2\ast]$
on $D^b_{\rm{coh}}(\mathfrak{U})$, and the right arrow is defined 
by taking Fourier-Mukai transforms associated with 
morphisms $\Delta_{\ast}\oO_{\mathfrak{U}} \to \Delta_{\ast}\oO_{\mathfrak{U}}[2\ast]$. 
The above maps induce the map for 
each $\fF \in D^b_{\rm{coh}}(\mathfrak{U})$, 
\begin{align}\notag
	\oO_{\mathrm{Crit}(w)}
	\to \Hom^{2\ast}(\fF, \fF).
\end{align}
The above map 
defines the 
$\mathbb{C}^{\ast}$-equivariant 
$\oO_{\mathrm{Crit}(w)}$-module 
structure on $\Hom^{2\ast}(\fF, \fF)$, 
which is finitely generated  by~\cite[Theorem~4.1.8]{MR3300415}. 
Below a closed subset $Z \subset \mathrm{Crit}(w)$ is called
\textit{conical} if it is closed under the fiberwise 
$\mathbb{C}^{\ast}$-action on $\mathrm{Crit}(w)$. 
	For $\fF \in D^b_{\rm{coh}}(\mathfrak{U})$, its singular
	support is the conical closed subset
	\begin{align*}
		\mathrm{Supp}^{\rm{sg}}(\fF) \subset \mathrm{Crit}(w)
	\end{align*}
	defined to be the support of $\Hom^{2\ast}(\fF, \fF)$
	as $\oO_{\mathrm{Crit}(w)}$-module. 

For a conical closed subset $Z \subset \Crit(w)$, let 
\begin{align*}
	\cC_{Z} \subset \Dbc(\fU), \ 
	\Ind \cC_{Z} \subset \Ind \Dbc(\fU)
	\end{align*}
be the triangulated subcategory consisting of objects
whose singular supports
are contained in $Z$, and its ind-completion respectively. 
Their dg-enhancements  
\begin{align*}
	\cC_{Z, \rm{dg}} \subset \Dbc(\fU)_{\rm{dg}}, \ 
	\Ind \cC_{Z, \rm{dg}} \subset \Ind \Dbc(\fU)_{\rm{dg}}
	\end{align*}
are also defined to be consisting of objects with singular 
supports contained in $Z$. 
\begin{prop}\label{prop:ssupp}\emph{(\cite[Proposition~2.3.9]{TocatDT})}
The equivalence $\Psi$ in Theorem~\ref{thm:knoer} 
restricts to the equivalences
\begin{align}\label{Psi:supp}
	\Psi \colon \MF^{\C}_{\coh}(V^{\vee}, w)_{Z} \stackrel{\sim}{\to} \cC_{Z}, \ 
	\Psi \colon \MF_{\qcoh}^{\C}(V^{\vee}, w)_{Z} \stackrel{\sim}{\to} \Ind \cC_{Z}. 
	\end{align}
In particular
by (\ref{equiv:supporZ}), the equivalences in Theorem~\ref{thm:knoer}
descend to the equivalences 
\begin{align}\label{Psi:supp2}
\Psi \colon \MF^{\C}_{\coh}(V^{\vee} \setminus Z, w)
\stackrel{\sim}{\to}
\Dbc(\fU)/\cC_{Z}, \ 
\Psi \colon \MF^{\C}_{\qcoh}(V^{\vee} \setminus Z, w)
\stackrel{\sim}{\to}
\Ind\Dbc(\fU)/\Ind\cC_{Z}. 
\end{align}
\end{prop}

\section{$\mathbb{Z}/2$-periodic Koszul duality}
In this section, we observe that a $\mathbb{Z}/2$-periodic analogue of Theorem~\ref{thm:knoer}
does not hold for the usual $\mathbb{Z}/2$-periodic 
derived categories, and need to use co (absolute) derived categories. 
We also use singular supports quotients to give another version of 
$\mathbb{Z}/2$-periodic analogue of Theorem~\ref{thm:knoer}. 

	\subsection{A failure of $\mathbb{Z}/2$-periodic Koszul duality}
	We will give a $\mathbb{Z}/2$-periodic version of 
	the Koszul duality equivalence in Theorem~\ref{thm:knoer}.
	Namely we would like to replace the $\C$-equivariant 
	factorization category $\MF_{\coh}^{\C}(V^{\vee}, w)$
	with the $\mathbb{Z}/2$-periodic 
	factorization category $\MF_{\coh}^{\mathbb{Z}/2}(V^{\vee}, w)$. 
	A naive guess is to replace 
$\Dbc(\fU)$ with the $\mathbb{Z}/2$-periodic derived category
\begin{align*}
	D^{\mathbb{Z}/2}_{\coh}(\fU) \cneq 
	D_{\rm{fg}}(\oO_{\fU}\modu^{\mathbb{Z}/2}). 
	\end{align*}
Let $\kK_s^{\mathbb{Z}/2}$ be the factorization $\kK_s$ of $w$
regarded as a $\mathbb{Z}/2$-graded factorization, i.e. 
\begin{align*}
	\kK_s^{\mathbb{Z}/2}=\kK_s^{\rm{even}} \oplus 
	\kK_s^{\rm{odd}}, \ 
	\kK_s^{\rm{even}}=\oO_{V^{\vee}} \otimes_{\oO_Y}
	\wedge^{\rm{even}}V^{\vee}, \ 
	\kK_s^{\rm{odd}}=\oO_{V^{\vee}} \otimes_{\oO_Y}\wedge^{\rm{odd}}V^{\vee},
	\end{align*}
with differential given by (\ref{diff:ds}). 
As an analogy of (\ref{equiv:Psi:K}), we can define the functor 
\begin{align}\label{Psi:half}
	\Psi^{\mathbb{Z}/2} \colon 
	\MF_{\coh}^{\mathbb{Z}/2}(V^{\vee}, w) \to 
	D_{\coh}^{\mathbb{Z}/2}(\fU), \ 
	(-) \mapsto 
	\RHom(\kK_s^{\mathbb{Z}/2}, -). 
	\end{align}
We see that the above functor does not necessary give an equivalence. 
	 \begin{exam}\label{exam:1}
	 	Let $Y=\Spec \mathbb{C}$, 
	 	$V=\mathbb{A}^1$, $s=0$ so that 
	 	$\fU=\bullet_{\epsilon} \cneq \Spec \mathbb{C}[\epsilon]$ 
	 	with $\deg(\epsilon)=-1$. 
	 	In this case, the Koszul factorization is 
	 	$\kK_s=\mathbb{C}[\epsilon, t]$ where 
	 	$\deg(t)=2$ with differential 
	 	given by the multiplication $\epsilon t$. 
	 	The functor (\ref{Psi:half}) is 
	 	\begin{align}\label{funct:A1}
	 		\Psi^{\mathbb{Z}/2} \colon \MF_{\coh}^{\mathbb{Z}/2}(\mathbb{A}^1, 0)
	 		\to D_{\coh}^{\mathbb{Z}/2}(\bullet_{\epsilon}), \ 
	 		(-) \mapsto  \RHom(\kK_s^{\mathbb{Z}/2}, -). 
	 		\end{align}
 		Then for $y \in \mathbb{A}^1$ we have 
 		\begin{align}\label{Psiy}
 			\Psi^{\mathbb{Z}/2}(\oO_y) =
 			\left( \mathbb{C} \oplus \mathbb{C} \epsilon^{\vee}, 
 			d=y\epsilon \right)
 			 \cong \begin{cases}
 				(\mathbb{C}[\epsilon])[1], & y=0, \\
 				0, & y \neq 0. 
 				\end{cases}
 			\end{align}
 		In particular 
 		$\Psi^{\mathbb{Z}/2}$ is not fully-faithful. 
 		
 			Geometrically this issue
 		occurs since the factorization $\kK_s^{\mathbb{Z}/2}$ 
 		is supported at the 
 		zero section of $V^{\vee} \to Y$ 
 		so that the functor (\ref{funct:A1}) annihilates objects whose 
 		supports are away from the zero section. 	
 		We note that the 
 		above issue does not happen for $\C$-equivariant case (as indicated 
 		by Theorem~\ref{thm:knoer}), 
 		since the support of any $\C$-invariant coherent sheaf of $V^{\vee}$ 
 		intersects with the zero section (so the objects $\oO_y$ for $y\neq 0$
 		are not allowed in the $\C$-equivariant case). 
 		\end{exam}
 	
 	\begin{exam}\label{exam:2}
 		In Example~\ref{exam:1}, the 
 		functor (\ref{funct:A1}) also 
 		sends $\oO_{\mathbb{A}^1}$ to $\oO_0 \cneq \mathbb{C}[\epsilon]/(\epsilon)$. 
 	We have the isomorphisms of $\mathbb{Z}/2$-graded vector spaces
 		\begin{align}\label{compute:e}
 			\Hom^{\ast}_{\MF_{\coh}^{\mathbb{Z}/2}(\mathbb{A}^1, 0)}(\oO_{\mathbb{A}^1}, \oO_{\mathbb{A}^1})
 			=\mathbb{C}[t], \ 
 				\Hom^{\ast}_{D_{\coh}^{\mathbb{Z}/2}(\bullet_{\epsilon})}(\oO_0, \oO_0)=
 				\mathbb{C}\lkakko t \rkakko
 			\end{align}
 		where $t$ is of even degree. 
 		The right hand side 
 		is computed by taking the projective resolution 
 		\begin{align}\notag
 		\cdots 	\to \mathbb{C}[\epsilon] \stackrel{\epsilon}{\to}
 			\mathbb{C}[\epsilon] \to \oO_0 \to 0.  			
 			\end{align}
 		As the above resolution exhibits $\oO_0$ as a colimit of 
 		$(\cdots 	\to \mathbb{C}[\epsilon] \stackrel{\epsilon}{\to}
 		\mathbb{C}[\epsilon])$, 
 		the $\Hom$ space 
 		in the right hand side of (\ref{compute:e}) is computed by the  
 		limit of $\mathbb{C} \stackrel{0}{\to} \mathbb{C} \stackrel{0}{\to} \cdots$
 		in the category of $\mathbb{Z}/2$-graded vector spaces, 
 		that is $\mathbb{C}\lkakko t \rkakko$. 
 		The above computation also implies that 
 		$\Psi^{\mathbb{Z}/2}$ is not fully-faithful. 
 		
 		In the $\C$-equivariant case,
 		we compute 
 		$\Hom^{\ast}_{D^{b}_{\coh}(\bullet_{\epsilon})}(\oO_0, \oO_0)$
 		as a limit of 
 		$\mathbb{C} \stackrel{0}{\to} \mathbb{C}[-2] \stackrel{0}{\to} \cdots$
 		in the category $\mathbb{Z}$-graded vector spaces, 
 		which is $\mathbb{C}[t]$ where $t$ is of degree two. 
 		Therefore we have 
 		\begin{align*}
 			\Hom^{\ast}_{\MF_{\coh}^{\C}(\mathbb{A}^1, 0)}(\oO_{\mathbb{A}^1}, \oO_{\mathbb{A}^1})
 			=\Hom^{\ast}_{D^{b}_{\coh}(\bullet_{\epsilon})}(\oO_0, \oO_0)
 			=\mathbb{C}[t].
 			\end{align*} 
 		\end{exam}
 	
 	\subsection{$\mathbb{Z}/2$-periodic Koszul duality via coderived categories}
 	A failure of the functor (\ref{Psi:half}) to be an equivalence 
 	is caused by the incorrect definition 
 	of the target category. 
 	Indeed in the $\mathbb{Z}/2$-periodic case, 
 	an analogue of Lemma~\ref{lem:equivDb} does not hold, 
 	so we need to distinguish two kinds of derived 
 	categories. 
 	Technically, a $\mathbb{Z}/2$-periodic analogue of the functor (\ref{equiv:Phi}) 
 	is not well-defined from $D^{\mathbb{Z}/2}_{\coh}(\fU)$, but well-defined from 
 	coderived category or absolute derived category. 
 	As we see below, the latter derived categories
 	provide a correct formulation of $\mathbb{Z}/2$-periodic 
 	Koszul duality. 
 	 	
 	\begin{prop}\label{prop:Z2dual}
 		In the setting of Theorem~\ref{thm:knoer}, 
 		there is an equivalence 
 		\begin{align}\label{equiv:Z21}
 			\Psi^{\mathbb{Z}/2} \colon \MF^{\mathbb{Z}/2}_{\qcoh}(V^{\vee}, w)
 			\stackrel{\sim}{\to} D^{\rm{co}}(\oO_{\fU}\modu^{\mathbb{Z}/2}). 
 			\end{align} 		
 		\end{prop} 
 	\begin{proof}
 		Let 
 		$\underline{\MF}^{\mathbb{Z}/2}_{\qcoh}(V^{\vee}, w)_{\rm{dg}}$ 
 		be the dg-category of $\mathbb{Z}/2$-periodic quasi-coherent 
 		factorizations (see the notation in Subsection~\ref{subsec:fact}). 
 		We have the dg functor
 		\begin{align*}
 			\Hom^{\ast}(\kK_s^{\mathbb{Z}/2}, -) \colon 
 			\underline{\MF}^{\mathbb{Z}/2}_{\qcoh}(V^{\vee}, w)_{\rm{dg}}
 			\to \oO_{\fU}\modu^{\mathbb{Z}/2}. 
 			\end{align*}
 		Here the $\oO_{\fU}$-module structure on 
 		$\Hom^{\ast}(\kK_s^{\mathbb{Z}/2}, -)$
 		is induced by the $\oO_{\fU}$-module structure on 
 		$\kK_s^{\mathbb{Z}/2}=\oO_{V^{\vee}} \otimes_{\oO_Y}
 		\oO_{\fU}$, i.e. the multiplication by the right factor. 
 		Since $\kK_s^{\mathbb{Z}/2}$ is projective as $\oO_{V^{\vee}}$-module, 
 		the above functor preserves absolutely acyclic objects in the homotopy categories. 
 		Therefore it induces the functor 
 		\begin{align*}
 			\Psi^{\mathbb{Z}/2} \colon \MF^{\mathbb{Z}/2}_{\qcoh}(V^{\vee}, w)
 			\to D^{\rm{co}}(\oO_{\fU}\modu^{\mathbb{Z}/2}). 
 			\end{align*}
 		Similarly we have the dg functor 
 		\begin{align*}
 			\kK_s^{\mathbb{Z}/2} \otimes_{\oO_{\fU}}(-) \colon 
 			\oO_{\fU}\modu^{\mathbb{Z}/2}
 			\to \underline{\MF}^{\mathbb{Z}/2}_{\qcoh}(V^{\vee}, w)_{\rm{dg}}. 
 			\end{align*}
 		Since $\kK_s^{\mathbb{Z}/2}$ is projective as a $\mathbb{Z}/2$-graded $\oO_{\fU}$-module, 
 		and taking tensor product preserves direct sums, 
 		the above functor preserves coacyclic objects. 
 		As $V^{\vee}$ is smooth, 
 		coacyclic objects in the homotopy 
 		category of $\mathrm{MF}^{\mathbb{Z}/2}_{\qcoh}(V^{\vee}, w)_{\rm{dg}}$
 		coincide with absolutely acyclic objects by~\cite[Section~3.6, Theorem]{Postwo}. 
 		Therefore the above functor induces the functor 
 		\begin{align*}
 			\Phi^{\mathbb{Z}/2} \colon 
 		 D^{\rm{co}}(\oO_{\fU}\modu^{\mathbb{Z}/2})
 		 \to
 			\MF^{\mathbb{Z}/2}_{\qcoh}(V^{\vee}, w). 
 			\end{align*}
 		It is enough to show that 
 		$\Psi^{\mathbb{Z}/2} \circ \Phi^{\mathbb{Z}/2} \cong \id$ and $\Phi^{\mathbb{Z}/2} \circ \Psi^{\mathbb{Z}/2} \cong \id$. 
 	
 	We first show the isomorphism 
 	$\Phi^{\mathbb{Z}/2} \circ \Psi^{\mathbb{Z}/2} \cong \id$.  		
 		By the definition of $\Phi^{\mathbb{Z}/2}$ and $\Psi^{\mathbb{Z}/2}$, 
 		we have 
 		\begin{align*}
 			\Phi^{\mathbb{Z}/2} \circ \Psi^{\mathbb{Z}/2}(-)=\kK_s^{\mathbb{Z}/2} \otimes_{\oO_{\fU}} \kK_s^{\mathbb{Z}/2\vee} \otimes_{\oO_{V^{\vee}}}(-). 
 			\end{align*}
 		Here $\kK_s^{\mathbb{Z}/2\vee}=\oO_{\fU}^{\vee} \otimes_{\oO_Y}\oO_{V^{\vee}}$ is the $\oO_{V^{\vee}}$-dual 
 		of $\kK_s^{\mathbb{Z}/2}$ which is a factorization of $-w$, 
 		where $\oO_{\fU}^{\vee}$ is the $\oO_Y$-dual of $\oO_{\fU}$, 
 		\begin{align*}
 			\oO_{\fU}^{\vee}=S_{\oO_Y}(V[-1])=\det(V)[-\rank(V)] \otimes_{\oO_Y}\oO_{\fU}. 
 			\end{align*}
 		We have 
 		\begin{align}\label{Kotimes}
 			\kK_s^{\mathbb{Z}/2} \otimes_{\oO_{\fU}} \kK_s^{\mathbb{Z}/2\vee}=
 			(\oO_{V^{\vee}} \otimes_{\oO_Y}\oO_{\fU}^{\vee} \otimes_{\oO_Y}
 			\oO_{V^{\vee}}, 1 \otimes d_{\oO_{\fU}^{\vee}} \otimes 1
 			+\eta_1 \otimes 1_{\oO_{V^{\vee}}}-1_{\oO_{V^{\vee}}}\otimes \eta_2).
 			\end{align}
 		Here $\eta_1 \in V\otimes_{\oO_Y}V^{\vee}$ and $\eta_2 \in V^{\vee} \otimes_{\oO_Y}
 		V$ correspond to $\id \in \Hom(V, V)$, which act
 		on $\oO_{V^{\vee}} \otimes_{\oO_Y}\oO_{\fU}^{\vee}$
 		and $\oO_{\fU}^{\vee} \otimes_{\oO_Y} \oO_{V^{\vee}}$ respectively. 
 		The above object (\ref{Kotimes}) is a factorization of
 		the function 
 		\begin{align*}
 			-p_1^{\ast}w+p_2^{\ast}w \colon 
 			V^{\vee} \times_Y V^{\vee} \to \mathbb{C},
 			\end{align*} giving a 
 		Fourier-Mukai kernel of 
 		$\Phi^{\mathbb{Z}/2} \circ \Psi^{\mathbb{Z}/2}$. 
 		Here $p_i \colon V^{\vee}\times_Y V^{\vee} \to V^{\vee}$ is the projection
 		onto the corresponding factor. 
 		The function $-p_1^{\ast}w+p_2^{\ast}w$ is explicitly written as
 		\begin{align}\label{w:ast}
 			(-p_1^{\ast}w+p_2^{\ast}w)(x, v_2, v_2)=\langle s(x), v_2-v_1\rangle
 			\end{align}
 		for $(x, v_1, v_2) \in V^{\vee}\times_Y V^{\vee}$
 		with $x\in Y$, $v_i \in V^{\vee}|_{y}$. 
 		Let $\alpha, \beta$ be maps
 		\begin{align*}
 			&\alpha \colon V^{\vee} \times_Y V^{\vee} \to 
 			 V^{\vee} \times_Y V^{\vee}\times_Y V^{\vee}, \ 
 			 (x, v_1, v_2) \mapsto (x, v_1, v_2, v_2-v_1), \\
 			 	&\beta \colon V^{\vee} \times_Y V^{\vee} \to 
 			 V^{\vee} \times_Y V^{\vee}\times_Y V, \ 
 			 (x, v_1, v_2) \mapsto (x, v_1, v_2, s(x)). 
 			\end{align*}
 		They are sections of 
 		the vector bundles  
 		\begin{align*}
 			V^{\vee} \times_Y V^{\vee}\times_Y V^{\vee} \to V^{\vee}\times_Y V^{\vee}, \
 		 V^{\vee} \times_Y V^{\vee}\times_Y V \to V^{\vee} \times_Y V^{\vee}
 		 \end{align*}
 	 over $V^{\vee} \times_Y V^{\vee}$, 
 	 given by projections onto the left two factors 
 		 respectively. 
 		 By (\ref{w:ast}), we have 
 		 $-p_1^{\ast}w+p_2^{\ast}w=\langle \alpha, \beta \rangle$. 
 			By unraveling the right hand side of (\ref{Kotimes}), 
 			we see that (\ref{Kotimes}) is the Koszul 
 			factorization associated with $(\alpha, \beta)$ (see (\ref{Kfact})). 
 			Since $\alpha$ is a regular section 
 			whose zero locus is the diagonal 
 			$\Delta \subset V^{\vee}\times_Y V^{\vee}$, by (\ref{Kos:isom})
 		 			we have an isomorphism 
 			in $\MF_{\qcoh}^{\mathbb{Z}/2}(V^{\vee}\times_Y V^{\vee}, 
 			-p_1^{\ast}w+p_2^{\ast}w)$, 
 			\begin{align*}
 				\kK_s^{\mathbb{Z}/2} \otimes_{\oO_{\fU}} \kK_s^{\mathbb{Z}/2\vee}\cong (\oO_{\Delta}, d_{\oO_{\Delta}}=0).
 				\end{align*}
 			The above isomorphism implies that $\Phi^{\mathbb{Z}/2} \circ \Psi^{\mathbb{Z}/2} \cong \id$. 
 			
 			We next show the isomorphism 
 			$\Psi^{\mathbb{Z}/2} \circ \Phi^{\mathbb{Z}/2} \cong \id$. 
 			We have 
 			\begin{align*}
 			\Psi^{\mathbb{Z}/2} \circ \Phi^{\mathbb{Z}/2}(-) =\kK_s^{\mathbb{Z}/2\vee} \otimes_{\oO_{V^{\vee}}}
 			\kK_s^{\mathbb{Z}/2} \otimes_{\oO_{\fU}}(-). 	
 				\end{align*}
 			Here $\kK_s^{\mathbb{Z}/2\vee} \otimes_{\oO_{V^{\vee}}}
 			\kK_s^{\mathbb{Z}/2}$ is 
 			\begin{align}\label{Kotimes2}
 				\kK_s^{\mathbb{Z}/2\vee} \otimes_{\oO_{V^{\vee}}} \kK_s^{\mathbb{Z}/2}=
 				(\oO_{\fU}^{\vee} \otimes_{\oO_Y} \oO_{V^{\vee}} \otimes_{\oO_Y}
 				\oO_{\fU}, d_{\oO_{\fU}^{\vee}}+d_{\oO_{\fU}}-\eta_2 \otimes 
 				1_{\oO_{\fU}}+
 				1_{\oO_{\fU}^{\vee}}\otimes \eta_1). 
 				\end{align}
 			The above object is a $\mathbb{Z}/2$-graded
 			dg $\oO_{\fU \times_Y \fU}$-module, which naturally 
 			lifts to a $\mathbb{Z}$-graded one 
 			$\kK_s^{\vee} \otimes_{\oO_{V^{\vee}}} \kK_s$. 
 		 			On the other hand, 
 			the automorphism of the vector bundle
 			$V \times_Y V$ on $Y$ by 
 			$(x, y) \mapsto (x+y, x-y)/2$ induces the
 			equivalence of derived schemes
 			\begin{align}\label{eq:Usta}
 				\fU \times_Y \fU \stackrel{\sim}{\to} \fU^{\flat} \times_Y 
 				\fU. 
 				\end{align}
 			Here $\fU^{\flat}=\Spec S_{\oO_Y}(V^{\vee}[1])$ with zero 
 			differential on $S_{\oO_Y}(V^{\vee}[1])$. 
 			Under the above equivalence, the 
 			object $\Delta_{\ast}\oO_{\fU}$ corresponds to 
 			$\oO_Y \boxtimes \oO_{\fU}$. 
 			We have the Koszul 
 			resolution of $\oO_{Y}$ as 
 			$\mathbb{Z}$-graded dg $\oO_{\fU^{\flat}}$-modules 
 			\begin{align*}
 				\cdots 
 				\to \bigwedge^2(V^{\vee}[1]) \otimes_{\oO_Y}
 				\oO_{\fU^{\flat}} 
 				\to V^{\vee}[1] \otimes_{\oO_Y} \oO_{\fU^{\flat}}
 				\to \oO_{\fU^{\flat}} \to \oO_Y \to 0. 
 				\end{align*}
 			By taking the $\oO_Y$-dual of the above resolution, 
 			we obtain the following resolution 
 			\begin{align*}
 				0 \to \oO_Y \to \oO_{\fU^{\flat}}^{\vee} \to 
 				\oO_{\fU^{\flat}}^{\vee} \otimes_{\oO_Y}V[-1] \to 
 				\oO_{\fU^{\flat}}^{\vee} \otimes_{\oO_Y} \bigwedge^2(V[-1]) \to 
 				\cdots. 
 				\end{align*}
 			The totalization of the above sequence is a $\mathbb{Z}$-graded
 			dg-module over $\oO_{\fU^{\flat}}$ which is acyclic and bounded 
 			below, hence it is coacyclic 
 			by~\cite[Section~3.4, Theorem~1 (a)]{Postwo}. 
 		By applying $\boxtimes \oO_{\fU}$ and pull-back 
 		by the equivalence (\ref{eq:Usta}), 
 		we obtain the complex of $\mathbb{Z}$-graded
 		dg $\oO_{\fU \times_Y \fU}$-modules 
 		\begin{align*}
 		0 \to \Delta_{\ast}\oO_{\fU} \to 
 		\oO_{\fU}^{\vee} \otimes_{\oO_Y} \oO_{\fU} \to 
 		\oO_{\fU}^{\vee} \otimes_{\oO_Y} V[-1] \otimes_{\oO_Y}
 		\oO_{\fU} \to 
 			\oO_{\fU}^{\vee} \otimes_{\oO_Y} \bigwedge^2(V[-1]) \otimes_{\oO_Y}
 		\oO_{\fU} \to \cdots 
 		\end{align*}
 	whose totalization is coacyclic. 
 	By unraveling the differentials and comparing with (\ref{Kotimes2}), 
 	we see that the above 
 	complex is $\Delta_{\ast}\oO_{\fU} \to 
 	\kK_s^{\vee}\otimes_{\oO_{V^{\vee}}}\kK_s$. 
 	It follows that 
 	we have the isomorphism 
 	\begin{align}\notag
 		\Delta_{\ast}\oO_{\fU} \stackrel{\cong}{\to}
 		\kK_s^{\vee}\otimes_{\oO_{V^{\vee}}}\kK_s
 		\end{align}
 	in $D^{\rm{co}}(\oO_{\fU \times_Y \fU}\modu^{\mathbb{Z}})$. 
 	Therefore the object (\ref{Kotimes2}) is also 
 	isomorphic to $\Delta_{\ast}\oO_{\fU}$ 
 	in  
 	$D^{\rm{co}}(\oO_{\fU \times_Y \fU}\modu^{\mathbb{Z}/2})$.
  	Since $\oO_{\fU}$ is Gorenstein,
  	by~\cite[Section~1.7, Proposition]{MR3366002}
 	an object in $D^{\rm{co}}(\oO_{\fU}\modu^{\mathbb{Z}/2})$
 	 is represented by a $\mathbb{Z}/2$-graded
 	dg $\oO_{\fU}$-module $M$ which is flat over $\oO_{\fU}$. 
 	For such an object $M$ and a coacyclic 
 	object $\pP \in \mathrm{Ho}(\oO_{\fU \times_Y \fU}\modu^{\mathbb{Z}/2})$, 
 	the object $M \otimes_{\oO_{\fU}}\pP$ is also coacyclic. 
 	Therefore
 	the isomorphism $\Psi^{\mathbb{Z}/2} \circ \Phi^{\mathbb{Z}/2}(M) \cong M$ holds. 
 		\end{proof}
 	 	
 	 	\begin{cor}\label{cor:abs}
 	 		There is an equivalence
 	 		\begin{align}\notag
 	 			\Psi^{\mathbb{Z}/2} \colon  \overline{\MF}^{\mathbb{Z}/2}_{\coh}(V^{\vee}, w)
 	 			\stackrel{\sim}{\to} \overline{D}^{\rm{abs}}(\oO_{\fU}\modu_{\rm{fg}}^{\mathbb{Z}/2}). 
 	 		\end{align}
 	 		Here $\overline{(-)}$ indicates idempotent completion. 
 	 		\end{cor}
  		\begin{proof}
  			The equivalence (\ref{equiv:Z21}) restricts to 
  			the equivalence between the subcategories of compact objects. 
  			Since 
  			$\mathrm{MF}_{\qcoh}^{\mathbb{Z}/2}(V^{\vee}, w)$ is compactly generated by 
  			$\mathrm{MF}_{\coh}^{\mathbb{Z}/2}(V^{\vee}, w)$ 
  			(see~\cite[Proposition~3.15]{MR3270588}), 
  			and 
  			$D^{\rm{co}}(\oO_{\fU}\modu^{\mathbb{Z}/2})$
  			is compactly generated by compact objects 
  			$D^{\rm{abs}}(\oO_{\fU}\modu_{\rm{fg}}^{\mathbb{Z}/2})$ (see~\cite[Section~3.11, Theorem~1]{Postwo}), we have 
  			equivalences
  			\begin{align*}
  				\overline{\MF}^{\mathbb{Z}/2}_{\coh}(V^{\vee}, w)
  				\stackrel{\sim}{\to}\mathrm{MF}_{\qcoh}^{\mathbb{Z}/2}(V^{\vee}, w)^{\rm{cp}}, \ 
  				\overline{D}^{\rm{abs}}(\oO_{\fU}\modu_{\rm{fg}}^{\mathbb{Z}/2})
  				\stackrel{\sim}{\to} 
  				D^{\rm{co}}(\oO_{\fU}\modu^{\mathbb{Z}/2})^{\rm{cp}}.
  				\end{align*}
  			Therefore the corollary follows. 
  			\end{proof}
  		
  		\begin{exam}\label{exam:3}
  			In the setting of Example~\ref{exam:1}, 
  			let us consider the functor 
  			\begin{align*}
  				\Psi^{\mathbb{Z}/2} \colon \mathrm{MF}_{\qcoh}^{\mathbb{Z}/2}(\mathbb{A}^1, 0)
  				\to D^{\rm{co}}(\mathbb{C}[\epsilon]\modu^{\mathbb{Z}/2}). 
  				\end{align*}
  			For $y\in \mathbb{A}^1$, 
  			we have \begin{align*}
  				\Psi^{\mathbb{Z}/2}(\oO_y)=
  				\left( \mathbb{C} \oplus \mathbb{C} \epsilon^{\vee}, 
  				d=y\epsilon 
  				\right).
  			\end{align*}
  		A difference from Example~\ref{exam:1} is that we no longer have 
  		the second isomorphism in (\ref{Psiy}) 
  		so that $\Psi^{\mathbb{Z}/2}(\oO_y) \neq 0$ in the coderived category. 
  			\end{exam}
  		
  		\begin{rmk}\label{rmk:Zgrade}
  			The proof of Proposition~\ref{prop:Z2dual} is also applied to the 
  			$\mathbb{Z}$-graded case, which shows the equivalence
  			\begin{align*}
  				\Psi \colon \MF_{\qcoh}^{\C}(V^{\vee}, w)
  				\stackrel{\sim}{\to} D^{\rm{co}}(\oO_{\fU}\modu^{\mathbb{Z}}). 
  				\end{align*}
  			As the right hand side is equivalent to $\Ind \Dbc(\fU)$
  			(see Remark~\ref{rmk:indcoh}), 
  			we recover the equivalence (\ref{ind:Psi}) in Theorem~\ref{thm:knoer}. 
  			\end{rmk}
 	 
 	 	\subsection{$\mathbb{Z}/2$-periodic Koszul duality via singular supports}\label{Kdual:supp}
 	 As we mentioned in Remark~\ref{rmk:indcoh}, the coderived category in 
 	 the $\mathbb{Z}$-graded case is equivalent to the category of 
 	 ind-coherent sheaves, whose general theory is well-established~\cite{MR3136100, MR3300415}
 	 in derived algebraic geometry.
 	 On the other hand, the $\mathbb{Z}/2$-periodic coderived category is 
 	 not yet well developed in the context of derived algebraic geometry, 
 	 e.g. its singular support theory and higher categorical treatment are missing. 
 	 In this subsection, we give another kind of $\mathbb{Z}/2$-periodic 
 	 Koszul duality described in the framework of usual $\mathbb{Z}$-graded
 	 derived category of coherent sheaves and ind-coherent sheaves. 
 	 The Koszul duality in this subsection will be used 
 	 to give a global model for $\mathbb{Z}/2$-periodic DT category. 
 	  We first prepare some lemmas. 
 	
	\begin{lem}\label{lem:quotient}
	Let $\C$ acts on a $\mathbb{C}$-scheme $A$ 
	which restricts to the trivial action on 
	$\mu_2 \subset \C$. 
	The inclusion $A \hookrightarrow A \times \C$, 
	$x \mapsto (x, 1)$ 
	induces the isomorphism of stacks 
	\begin{align}\label{isom:A2}
	\iota \colon [A/\mu_2] \stackrel{\cong}{\to} [(A \times \C)/\C]. 
		\end{align}
	Here $\mu_2$ acts on $A$ trivially, and 
	$\C$ acts on $A \times \C$ by 
	$t(x, u)=(t(x), t^2 u)$. 
		\end{lem}
	\begin{proof}
		By the assumption that $\C$-action on $A$ restricts to the 
		trivial $\mu_2$-action, 
		the map 
		$\C \times A \to A$, $(x, u) \mapsto u^{-1/2} \cdot x$
		is well-defined and $\C$-invariant. 
		By taking the quotient by $\C$, 
		 we obtain the morphism of stacks 
		$[(A \times \C)/\C] \to A$. 
		We also have the projection 
		$[(A \times \C)/\C] \to [\C/\C]=B \mu_2$, 
		so we obtain the morphism 
		\begin{align*}
			[(A \times \C)/\C] \to A \times B{\mu_2}=[A/\mu_2]. 
			\end{align*}
		It is straightforward to check that the above morphism 
		gives an inverse of (\ref{isom:A2}). 
		\end{proof}
	
	\begin{lem}\label{lem:equivA}
		In the setting of Lemma~\ref{lem:quotient}, suppose that $A$ is smooth. 
		Let  
		$w \colon A \to \mathbb{C}$ be a regular function of $\C$-weight two, 
		and $w_{\epsilon} \colon A \times \C \to \mathbb{C}$ be 
		defined by $w_{\epsilon}(x, t)=w(x)$. 
		Then we have equivalences
		\begin{align*}
			\iota^{\ast} \colon 
			  \MF^{\C}_{\star}(A \times \C, w_{\epsilon})
			 \stackrel{\sim}{\to} 
			 \MF^{\mathbb{Z}/2}_{\star}(A, w), \ 
			 \star \in \{\qcoh, \coh\}. 
			\end{align*}
		\end{lem}
	\begin{proof}
		The function $w_{\epsilon}$ is of $\C$-weight two, 
		so it is a global section of 
		the line bundle $\oO(2)$ on 
		$[(A \times \C)/\C]$ induced by the $\C$-character of weight two. 
		By the isomorphism $\iota$ in Lemma~\ref{lem:quotient}, it is 
		pulled back to the global section $w$ of the trivial line bundle on 
		$[A/\mu_2]$. 
		Therefore the lemma follows from Lemma~\ref{lem:quotient}. 
		\end{proof}
	
	We return to the setting of Theorem~\ref{thm:knoer}.
	We define the following affine derived scheme
	\begin{align*}
		\fU_{\epsilon} \cneq 
		\fU \times \Spec \mathbb{C}[\epsilon],
		\end{align*} 
	where $\deg(\epsilon)=-1$ with zero differential. 
	Note that $\fU_{\epsilon}$ is the derived zero 
	locus of $s_{\epsilon} \cneq (s, 0) \colon Y \to V \times \mathbb{A}^1$
	for the vector bundle $V\times \mathbb{A}^1 \to Y$. 
	Therefore we have 
	\begin{align*}
		t_0(\Omega_{\fU_{\epsilon}}[-1])=\Crit(w_{\epsilon})=
		\Crit(w) \times \mathbb{A}^1. 
		\end{align*}
	Here $w_{\epsilon} \colon V^{\vee} \times \mathbb{A}^1 \to \mathbb{C}$
	is given by $(x, v, t) \mapsto w(x, v)=\langle s(x), v \rangle$
	as in Lemma~\ref{lem:equivA}. 
	The following is another $\mathbb{Z}/2$-periodic version of Koszul duality equivalence: 
	\begin{prop}\label{prop:Koszul:Z2}
		There is an equivalence of triangulated categories
		\begin{align}\label{equiv:Z2}
		\Psi_{\epsilon} \colon 
		\MF_{\coh}^{\mathbb{Z}/2}(V^{\vee}, w) \stackrel{\sim}{\to}
		\Dbc(\fU_{\epsilon})/\cC_{\Crit(w) \times \{0\}}.
		\end{align}		
		\end{prop}
	\begin{proof}
		Let $\C$ acts on $V^{\vee} \times \mathbb{A}^1$
		by weight two on fibers of 
		$V^{\vee} \times \mathbb{A}^1 \to Y$. 
		By Lemma~\ref{lem:equivA}, 
		the left hand side of (\ref{equiv:Z2}) is equivalent to 
		$\MF^{\C}_{\coh}(V^{\vee} \times \C, w_{\epsilon})$. 
		We have the equivalence by (\ref{equiv:supporZ})
		\begin{align}\label{equiv:MF1}
			\MF^{\C}_{\coh}(V^{\vee} \times \mathbb{A}^1, w_{\epsilon})/
			\MF^{\C}_{\coh}(V^{\vee} \times \mathbb{A}^1, w_{\epsilon})_{V^{\vee} \times \{0\}}
			\stackrel{\sim}{\to}
			\MF^{\C}_{\coh}(V^{\vee} \times \C, w_{\epsilon}). 
			\end{align}
		By Theorem~\ref{thm:knoer}, we have the equivalence
		\begin{align*}
			\Psi \colon \MF^{\C}_{\coh}(V^{\vee} \times \mathbb{A}^1, w_{\epsilon})
			\stackrel{\sim}{\to}
			\Dbc(\fU_{\epsilon}). 				
			\end{align*}
		By (\ref{Psi:supp}),
		the above equivalence restricts to the equivalence 
		\begin{align*}
			\MF_{\coh}^{\C}(V^{\vee} \times \mathbb{A}^1)_{V^{\vee} \times \{0\}}
			\stackrel{\sim}{\to}
				\MF_{\coh}^{\C}(V^{\vee} \times \mathbb{A}^1)_{\Crit(w) \times \{0\}}
			\stackrel{\sim}{\to}	\cC_{\Crit(w) \times \{0\}}.
						\end{align*} 
		Here the first equivalence follows from 
		$\Crit(w_{\epsilon})\cap (V^{\vee} \times \{0\})=\Crit(w) \times \{0\}$
		and the equivalence (\ref{Crit:supp}). 
		By taking the Verdier quotients, we obtain the equivalence
		\begin{align}\label{equiv:MF2}
				\MF^{\C}_{\coh}(V^{\vee} \times \mathbb{A}^1, w_{\epsilon})/
			\MF^{\C}_{\coh}(V^{\vee} \times \mathbb{A}^1, w_{\epsilon})_{V^{\vee} \times \{0\}}
			\stackrel{\sim}{\to}
				\Dbc(\fU_{\epsilon})/\cC_{\Crit(w) \times \{0\}}.
			\end{align}
		By combining the above equivalences, 
		we obtain the equivalence (\ref{equiv:Z2}). 
		\end{proof}
	
	Let $\Phi_{\epsilon}$ be a quasi-inverse of $\Psi_{\epsilon}$
		\begin{align}\notag
			\Phi_{\epsilon} \colon 
					\Dbc(\fU_{\epsilon})/\cC_{\Crit(w) \times \{0\}}
					\stackrel{\sim}{\to}
						\MF_{\coh}^{\mathbb{Z}/2}(V^{\vee}, w). 
		\end{align}	
	It is described as follows: 
	\begin{lem}\label{lem:Phie}
		The equivalence $\Phi_{\epsilon}$ is given by 
	\begin{align}\label{Phie:write}
		\Phi_{\epsilon}(M, d_M)=(\oO_{V^{\vee}} \otimes_{\oO_Y}M^{\mathbb{Z}/2}, 
		1\otimes (d_M+\epsilon) +\eta). 
		\end{align}
	Here $(M, d_M)$ is a $\mathbb{Z}$-graded dg $\oO_{\fU_{\epsilon}}$-module, 
	$M^{\mathbb{Z}/2}=M^{\rm{even}} \oplus M^{\rm{odd}}$, 
	and $\eta \in V \otimes V^{\vee}$ corresponds to $\id \in \Hom(V, V)$
	as in (\ref{diff:ds}). 
	\end{lem}
\begin{proof}
	By the construction of $\Psi_{\epsilon}$, 
	its quasi-inverse $\Phi_{\epsilon}$ is a descendant of the 
	composition 
		\begin{align*}
		\Dbc(\fU_{\epsilon}) \stackrel{\sim}{\to}
		\MF_{\coh}^{\C}(V^{\vee} \times \mathbb{A}^1, w_{\epsilon}) \to 
		\MF_{\coh}^{\C}(V^{\vee}\times \C, w_{\epsilon}) 
		\stackrel{\sim}{\to} \MF_{\coh}^{\mathbb{Z}/2}(V^{\vee}, w). 
	\end{align*}
	Here the first equivalence is 
	given by Theorem~\ref{thm:knoer}, 
	the second functor is given by the restriction 
	to the open subset
	$V^{\vee} \times \C \subset V^{\vee} \times \mathbb{A}^1$, 
	and the last equivalence is given in Lemma~\ref{lem:equivA}. 
	The above composition is given by 
	\begin{align*}
		M \mapsto ((\kK_s \otimes_{\mathbb{C}} \mathbb{C}[\epsilon, t^{\pm 1}])
		\otimes_{\oO_{\fU_{\epsilon}}}M)|_{t=1} 
		=\oO_{V^{\vee}} \otimes_{\oO_Y} M^{\mathbb{Z}/2}. 
	\end{align*}
	Here $t$ has degree two
	and $\mathbb{C}[\epsilon, t^{\pm 1}]$ has differential 
	$t\epsilon$. Under the above isomorphism, the differential on 
	$\oO_{V^{\vee}} \otimes_{\oO_Y} M^{\mathbb{Z}/2}$ is given by 
	$1 \otimes (d_{M}+\epsilon) +\eta$. 
	Therefore the lemma holds. 
	\end{proof}

	\begin{exam}\label{exam:2.5}
		In the setting of Example~\ref{exam:1}, 
		we have $\fU=\Spec \mathbb{C}[\epsilon_1]$, 
		$\fU_{\epsilon}=\Spec \mathbb{C}[\epsilon_1, \epsilon_2]$
		where $\deg(\epsilon_i)=-1$, and $\Crit(w_{\epsilon})=\mathbb{A}^2$. 
		The equivalence in Proposition~\ref{prop:Koszul:Z2}
		is an equivalence
		\begin{align*}
			\Psi_{\epsilon} \colon \MF_{\coh}^{\mathbb{Z}/2}(\mathbb{A}^1, 0)
			\stackrel{\sim}{\to} \Dbc(\fU_{\epsilon})/\cC_{\mathbb{A}^1 \times \{0\}}. 
			\end{align*}
		For $y \in \mathbb{A}^1$, 
		let $\eE_y$ be defined by
		\begin{align*}
			\eE_y \cneq \mathbb{C}[\epsilon_1, \epsilon_2]/(\epsilon_2+y\epsilon_1)
			\in \Dbc(\fU_{\epsilon}). 
			\end{align*}
		Then from (\ref{Phie:write}),
		we have the isomorphism in $\MF_{\coh}^{\mathbb{Z}/2}(\mathbb{A}^1, 0)$
		\begin{align*}
			\Phi_{\epsilon}(\eE_y)=(\mathbb{C}[\epsilon, t], d_{\mathbb{C}[\epsilon, t]}=(t-y)\epsilon)
			\cong (\oO_y, d_{\oO_y}=0)[1]. 
			\end{align*} 
		Therefore $\Psi_{\epsilon}(\oO_y)=\eE_y[1]$. 
		The object $\eE_y$ has singular support 
	$\mathbb{A}^1(y, 1) \subset \mathbb{A}^2$, 
	so in particular it is non-zero in 
	$\Dbc(\fU_{\epsilon})/\cC_{\mathbb{A}^1 \times \{0\}}$. 
		\end{exam}
	
	We can also compare with the $\C$-equivariant Koszul duality. 
	Let $i \colon \fU \to \fU_{\epsilon}$ 
	be the morphism induced by 
	$\Spec \mathbb{C} \to \Spec \mathbb{C}[\epsilon]$. 
	We have the following lemma: 
	\begin{lem}\label{lem:compare}
		The following diagram is commutative 
		\begin{align}\label{compare:commute}
			\xymatrix{
	\Dbc(\fU) \ar[r]^-{\Phi} \ar[d]_-{i_{\ast}} & \MF_{\coh}^{\C}(V^{\vee}, w) 
	\ar[d]^-{\mathrm{forg}} \\		
		\Dbc(\fU_{\epsilon})/\cC_{\Crit(w) \times \{0\})}
		\ar[r]^-{\Phi_{\epsilon}} & \MF_{\coh}^{\mathbb{Z}/2}(V^{\vee}, w). 
	}
			\end{align}
		Here $\mathrm{forg}$ is the functor forgetting the $\C$-equivariant structure. 
				\end{lem}
			\begin{proof}
				For $M \in \Dbc(\fU)$, the element $\epsilon \in \mathbb{C}[\epsilon]$ acts 
				trivially on $i_{\ast}M$. Therefore the 
				commutativity of (\ref{compare:commute}) is obvious from (\ref{Phie:write}). 
				\end{proof}
	
			As mentioned in~\cite[Proposition~2.1]{Tudor3}, the results of
		Proposition~\ref{prop:Koszul:Z2} and Lemma~\ref{lem:compare} imply 
		that we have the isomorphism of $K$-groups
		of triangulated categories of $\C$-equivariant factorizations and $\mathbb{Z}/2$-periodic 
		ones. Although we will not use the corollary below, we include it here 
		for an independent interest: 
		\begin{cor}\label{cor:Kgp}
			The forgetting functor induces the isomorphism of $K$-groups
		\begin{align*}
			\mathrm{forg} \colon 
			K(\MF_{\coh}^{\C}(V^{\vee}, w)) \stackrel{\cong}{\to} K(\MF_{\coh}^{\mathbb{Z}/2}(V^{\vee}, w)). 
			\end{align*}
		\end{cor}
	\begin{proof}
		By Lemma~\ref{lem:compare}, it is enough to show 
		that the morphism induced by the left vertical arrow in (\ref{compare:commute})
		\begin{align*}
			i_{\ast} \colon K(\Dbc(\fU)) \to K(\Dbc(\fU_{\epsilon})/\cC_{\Crit(W) \times \{0\}})
			\end{align*}
		is an isomorphism. 
		Since $\Dbc(\fU_{\epsilon})$ and $\cC_{\Crit(w) \times \{0\}}$ are idempotent 
		complete, we have the exact sequence (see the argument of~\cite[Lemma~1.10]{PavShin})
		\begin{align}\label{exact:K}
			K(\cC_{\Crit(w) \times \{0\}}) \to K(\Dbc(\fU_{\epsilon})) \to K(\Dbc(\fU_{\epsilon})/\cC_{\Crit(W) \times \{0\}}) \to 0. 
			\end{align}
		Since $\fU$ and $\fU_{\epsilon}$ have isomorphic classical truncation $\uU$, 
		we have the isomorphisms
		\begin{align*}
			K(\Coh(\uU)) \stackrel{\cong}{\to} K(\Dbc(\fU)) \stackrel{\cong}{\to} K(\Dbc(\fU_{\epsilon})) 
			\end{align*}
		where the last isomorphism is induced by $i_{\ast}$. 
		Therefore it is enough to show that the first morphism in (\ref{exact:K}) is a zero map. 
		The subcategory 
		$\cC_{\Crit(w) \times \{0\}}$ is split generated by  
		$F \boxtimes C[\epsilon]$ for $F\in \Dbc(\fU)$, so in particular any 
		object $M \in C_{\Crit(w) \times \{0\}}$
		 satisfies $i^{\ast}M \in \Dbc(\fU)$. 
		Then from the triangle 
		\begin{align*}
			i_{\ast}i^{\ast}M[1] \to M \to i_{\ast}i^{\ast}M
			\end{align*}
		in $\Dbc(\fU_{\epsilon})$, we conclude $[M]=0$
		in $K(\Dbc(\fU_{\epsilon}))$.		
		\end{proof}

	By taking the ind-completion of the equivalence in Proposition~\ref{prop:Koszul:Z2}, 
	we have the following corollary: 
	\begin{cor}\label{cor:Z2ind}
		The equivalence (\ref{equiv:Z2}) extends to the equivalence
		\begin{align*}
			\Psi_{\epsilon} \colon 
			\MF_{\qcoh}^{\mathbb{Z}/2}(V^{\vee}, w) \stackrel{\sim}{\to}
			\Ind \Dbc(\fU_{\epsilon})/\Ind \cC_{\Crit(w)\times \{0\}}. 
			\end{align*}
		\end{cor}
	\begin{proof}
		Since $\MF_{\qcoh}^{\mathbb{Z}/2}(V^{\vee}, w)$ is compactly 
		generated by compact objects $\MF_{\coh}^{\mathbb{Z}/2}(V^{\vee}, w)$
		(see~\cite[Proposition~3.15]{MR3270588}), 
		we have the equivalence
		\begin{align*}
			\Ind \MF_{\coh}^{\mathbb{Z}/2}(V^{\vee}, w) \stackrel{\sim}{\to}
			\MF_{\qcoh}^{\mathbb{Z}/2}(V^{\vee}, w). 
			\end{align*}
		We also have the natural equivalence (see~\cite[Proposition~3.2.7]{TocatDT})
		\begin{align*}
			\Ind \Dbc(\fU_{\epsilon})/\Ind \cC_{\Crit(w)\times \{0\}}
			\stackrel{\sim}{\to} \Ind \left( \Dbc(\fU_{\epsilon})/\cC_{\Crit(w) \times \{0\}}  \right). 
			\end{align*}
		The corollary now follows by taking the ind-completion of the equivalence (\ref{equiv:Z2}). 
		\end{proof}
	
	Let $Z \subset \Crit(w)$ be a closed subset. 
	We define $Z_{\epsilon} \subset \Crit(w) \times \mathbb{A}^1$
	to be 
	\begin{align*}
		Z_{\epsilon} \cneq \C(Z \times \{1\}) \cup (\Crit(w) \times \{0\}). 
		\end{align*}
	Note that $Z_{\epsilon}$ is a conical closed subset in 
	$\Crit(w) \times \mathbb{A}^1$. 
	\begin{prop}\label{prop:Z2:supp}
		The equivalence in Proposition~\ref{prop:Koszul:Z2}
		descends to the equivalence 
		\begin{align}\label{equiv:descend}
			\Psi_{\epsilon} \colon \MF_{\coh}^{\mathbb{Z}/2}(V^{\vee} \setminus Z, w)\stackrel{\sim}{\to}
			\Dbc(\fU_{\epsilon})/\cC_{Z_{\epsilon}}. 
			\end{align}
				\end{prop}
	\begin{proof}
	The closed substack 
	$[Z/\mu_2] \subset [V^{\vee}/\mu_2]$ 
	corresponds to 
	$\C(Z \times \{1\}) \subset [(V^{\vee} \times \C)/\C]$ under the 
	isomorphism in Lemma~\ref{lem:quotient}.
	It follows that the equivalence in Lemma~\ref{lem:equivA} restricts to the equivalence 
	\begin{align*}
		\MF_{\coh}^{\C}(V^{\vee} \times \C, w_{\epsilon})_{\C(Z \times \{1\})} \stackrel{\sim}{\to}
		\MF_{\coh}^{\mathbb{Z}/2}(V^{\vee}, w)_{Z}. 
		\end{align*} 	
		The equivalence (\ref{equiv:MF1}) restricts to the equivalence 
		\begin{align*}
			\MF_{\coh}^{\C}(V^{\vee} \times \mathbb{A}^1, w_{\epsilon})_{Z_{\epsilon}}/
				\MF_{\coh}^{\C}(V^{\vee} \times \mathbb{A}^1, w_{\epsilon})_{\Crit(w) \times \{0\}}
				\stackrel{\sim}{\to}
					\MF_{\coh}^{\C}(V^{\vee} \times \C)_{\C(Z \times \{1\})}. 
			\end{align*}
		By (\ref{Psi:supp}), the equivalence (\ref{equiv:MF2}) restricts to the equivalence 
		\begin{align*}
			\MF_{\coh}^{\C}(V^{\vee} \times \mathbb{A}^1, w_{\epsilon})_{Z_{\epsilon}}/
		\MF_{\coh}^{\C}(V^{\vee} \times \mathbb{A}^1, w_{\epsilon})_{\Crit(w) \times \{0\}}
		\stackrel{\sim}{\to}
		\cC_{Z_{\epsilon}}/\cC_{\Crit{w} \times \{0\}}. 	
			\end{align*}
		By combining the above equivalences, the equivalence in Proposition~\ref{prop:Koszul:Z2}
		restricts to the equivalence
		\begin{align*}
			\Psi_{\epsilon} \colon 
			\MF_{\coh}^{\mathbb{Z}/2}(V^{\vee}, w)_Z \stackrel{\sim}{\to}
			\cC_{Z_{\epsilon}}/\cC_{\Crit(w) \times \{0\}}. 
			\end{align*}
		By taking the Verdier quotients and using (\ref{equiv:supporZ}), the equivalence in Proposition~\ref{prop:Koszul:Z2}
		descends to the equivalence (\ref{equiv:descend}). 
		\end{proof}
	
	\subsection{Comparison of two Koszul dualities}
	We compare the equivalences in Proposition~\ref{prop:Z2dual} and 
	Corollary~\ref{cor:Z2ind}. 
	Let $\Upsilon$ be the functor 
	\begin{align*}
		\Upsilon \colon \mathrm{Ho}(\oO_{\fU_{\epsilon}}\modu^{\mathbb{Z}})
		\to \mathrm{Ho}(\oO_{\fU}\modu^{\mathbb{Z}/2}), \ 
		(M, d_M) \mapsto (M^{\mathbb{Z}/2}, d_M+\epsilon). 
		\end{align*}
	Here $M^{\mathbb{Z}/2}=M^{\rm{even}} \oplus M^{\rm{odd}}$ is a 
	$\mathbb{Z}/2$-graded 
	$\oO_{\fU}$-module forgetting the 
	$\mathbb{C}[\epsilon]$-module structure, 
	with differential $d_M+\epsilon$. 
	Since the above functor preserves coacyclic objects, it induces the functor
	\begin{align}\label{funct:Ups}
		\Upsilon \colon 
		D^{\rm{co}}(\oO_{\fU_{\epsilon}}\modu^{\mathbb{Z}}) \to 
		D^{\rm{co}}(\oO_{\fU}\modu^{\mathbb{Z}/2}). 
		\end{align}
	Note that the source of the above functor is equivalent to 
	$\Ind \Dbc(\fU_{\epsilon})$ 
	by Remark~\ref{rmk:indcoh}.
	\begin{lem}\label{lem:Upsilon}
		The functor (\ref{funct:Ups}) descends to the functor  
	\begin{align}\label{funct:Ups2}
		\Upsilon \colon \Ind \Dbc(\fU_{\epsilon})/\Ind \cC_{\Crit(w) \times \{0\}}
		\to D^{\rm{co}}(\oO_{\fU}\modu^{\mathbb{Z}/2}). 
		\end{align}
	\end{lem}
\begin{proof}
	Note that $\cC_{\{0\}} \subset \Dbc(\Spec \mathbb{C}[\epsilon])$ 
	coincides with the subcategory of perfect complexes on $\Spec \mathbb{C}[\epsilon]$ (see~\cite[Theorem~4.2.6]{MR3300415}), 
	which is generated by $\mathbb{C}[\epsilon]$. 
	Since $\Ind \Dbc(\fU_{\epsilon})$ is generated by exterior products 
	(see~\cite[Lemma~4.6.4]{MR3300415}), 
	it is enough to show that 
	the functor (\ref{funct:Ups}) sends objects 
	$M \boxtimes \mathbb{C}[\epsilon]$ for $M \in \Dbc(\fU)$
	to zero. 
	By the definition of $\Upsilon$, 
	we have 
	\begin{align*}
		\Upsilon(M \boxtimes \mathbb{C}[\epsilon])=
		\mathrm{Tot}(M^{\mathbb{Z}/2} \stackrel{\id}{\to}M^{\mathbb{Z}/2})
		\cong 0. 
		\end{align*}
	Therefore the lemma holds. 
	\end{proof}

As a summary of the results of this section, we have the following: 
\begin{thm}\label{thm:Z2kos}
	We have the commutative diagram of equivalences 
	\begin{align}\label{dia:MFind}
		\xymatrix{
\MF_{\qcoh}^{\mathbb{Z}/2}(V^{\vee}, w) \ar[d]_-{\Psi_{\epsilon}}^-{\sim} \ar@{=}[r] & 
	\MF_{\qcoh}^{\mathbb{Z}/2}(V^{\vee}, w)	\ar[d]^-{\Psi^{\mathbb{Z}/2}}_-{\sim} \\
 \Ind \Dbc(\fU_{\epsilon})/\Ind \cC_{\Crit(w) \times \{0\}} \ar[r]^-{\Upsilon}_-{\sim} & 
 D^{\rm{co}}(\oO_{\fU}\modu^{\mathbb{Z}/2}), 	
}
		\end{align}
	which restricts to the commutative diagram 
	\begin{align}\label{dia:MFind2}
		\xymatrix{
			\MF_{\coh}^{\mathbb{Z}/2}(V^{\vee}, w) \ar[d]_-{\Psi_{\epsilon}}^-{\sim} \ar@<-0.3ex>@{^{(}->}[r] & 
			\overline{\MF}_{\coh}^{\mathbb{Z}/2}(V^{\vee}, w)	\ar[d]^-{\Psi^{\mathbb{Z}/2}}_-{\sim} \\
			\Dbc(\fU_{\epsilon})/\cC_{\Crit(w) \times \{0\}} \ar@<-0.3ex>@{^{(}->}[r]^-{\Upsilon} & 
			\overline{D}^{\rm{abs}}(\oO_{\fU}\modu_{\rm{fg}}^{\mathbb{Z}/2}). 	
		}
	\end{align}
Here the horizontal arrows are fully-faithful with dense images. 
	\end{thm}
	\begin{proof}
	The commutativity of the diagram (\ref{dia:MFind}) follows from Lemma~\ref{lem:Phie}
	together with the isomorphism 
		\begin{align*}
			(\oO_{V^{\vee}} \otimes_{\oO_Y} M^{\mathbb{Z}/2}, 1\otimes (d_{M}+\epsilon)
			+\eta) \cong \kK_{s}^{\mathbb{Z}/2} \otimes_{\oO_{\fU}}(M^{\mathbb{Z}/2}, d_{M}+\epsilon)
			\end{align*}
		for $(M, d_M) \in \Dbc(\fU_{\epsilon})$. 
		Then the functor (\ref{funct:Ups2}) is an equivalence 
		since the vertical arrows in the diagram (\ref{dia:MFind}) are equivalences 
		by Proposition~\ref{prop:Z2dual} and Corollary~\ref{cor:Z2ind}. 
		The commutative diagram (\ref{dia:MFind2}) follows from Corollary~\ref{cor:abs} and
		Proposition~\ref{prop:Koszul:Z2}. 
			\end{proof}
	
	\section{$\mathbb{Z}/2$-periodic DT categories for $(-1)$-shifted cotangents}
	In this section, we give a basic model of $\mathbb{Z}/2$-periodic 
	DT categories for $(-1)$-shifted cotangents over quasi-smooth 
	derived stacks, based on the $\mathbb{Z}/2$-periodic 
	Koszul duality in the previous section. We then compare it 
	with the $\C$-equivariant DT category 
	studied in~\cite{TocatDT}. 
	
	\subsection{Quasi-smooth derived stacks}\label{subsec:qsmooth}
		Below, we denote by $\mathfrak{M}$ a
	derived Artin stack over $\mathbb{C}$.
	This means that 
	$\mathfrak{M}$ is a contravariant
	$\infty$-functor from 
	the $\infty$-category of 
	affine derived schemes over $\mathbb{C}$ to 
	the $\infty$-category of 
	simplicial sets 
	\begin{align*}
		\mathfrak{M} \colon 
		dAff^{op} \to SSets
	\end{align*}
	satisfying some conditions (see~\cite[Section~3.2]{MR3285853} for details). 
	Here $dAff^{\rm{op}}$ is defined to be the
	$\infty$-category of 
	commutative simplicial $\mathbb{C}$-algebras, 
	which is equivalent to the $\infty$-category of 
	commutative differential graded 
	$\mathbb{C}$-algebras with non-positive degrees. 
	The classical truncation of $\mathfrak{M}$ is denoted by 
	\begin{align*}
		\mM \cneq t_0(\mathfrak{M}) \colon 
		Aff^{op} \hookrightarrow 
		dAff^{op} \to SSets
	\end{align*}
	where the first arrow is a natural functor 
	from the category of affine schemes
	to affine derived schemes. 

For an affine derived scheme $\mathfrak{U}=\Spec R$
for a cdga $R$, 
we set $D_{\qcoh}(\fU)_{\rm{dg}} \cneq D(R\modu^{\Gamma})_{\rm{dg}}$. 	
The
	dg-category of quasi-coherent sheaves on $\mathfrak{M}$
	is defined
	to be the limit in the $\infty$-category of dg-categories (see~\cite{MR3285853})
	\begin{align}\notag
		D_{\rm{qcoh}}(\mathfrak{M})_{\rm{dg}} \cneq
		\lim_{\mathfrak{U} \to \mathfrak{M}} D_{\rm{qcoh}}(\mathfrak{U})_{\rm{dg}}.
	\end{align}
Here the limit is taken for the 
	$\infty$-category of smooth morphisms 
	$\alpha \colon \fU \to \fM$ for affine 
	derived schemes $\fU$ with 
	1-morphisms given by 
	smooth morphisms $\fU \to \fU'$ commuting with maps to $\fM$. 
	The homotopy category of $D_{\rm{qcoh}}(\mathfrak{M})_{\rm{dg}}$
	is denoted by 
	$D_{\rm{qcoh}}(\mathfrak{M})$, which is a 
	triangulated category. We have the dg and
	triangulated subcategories
	\begin{align*}
		D^b_{\rm{coh}}(\mathfrak{M})_{\rm{dg}} \subset 
		D_{\rm{qcoh}}(\mathfrak{M})_{\rm{dg}}, \ 
		D^b_{\rm{coh}}(\mathfrak{M}) \subset D_{\rm{qcoh}}(\mathfrak{M})
	\end{align*}
	consisting of objects which have bounded coherent 
	cohomologies.

	A morphism of derived stacks $f \colon \fM \to \fN$ is 
	called \textit{quasi-smooth} if 
	$\mathbb{L}_f$ is perfect 
	such that for any point $x \to \mM$
	the restriction $\mathbb{L}_f|_{x}$ is 
	of cohomological 
	amplitude $[-1, 1]$.
	Here 
	$\mathbb{L}_f$ is the $f$-relative cotangent complex. 
	A derived stack 
	$\mathfrak{M}$ over $\mathbb{C}$
	is called \textit{quasi-smooth}
	if $\fM \to \Spec \mathbb{C}$ is quasi-smooth. 
	By~\cite[Theorem~2.8]{MR3352237}, 
	the quasi-smoothness of $\mathfrak{M}$ is equivalent to 
	that $\fM$ 
	is a 1-stack, 
	and 
	any point of $\mathfrak{M}$ lies 
	in the image of a $0$-representable 
	smooth morphism 
	$\alpha \colon \mathfrak{U} \to \mathfrak{M}$, 
	where $\mathfrak{U}$ is an affine derived scheme 
	obtained as a derived zero locus as 
	in (\ref{frak:U}). 
	In this case, we have 
	\begin{align*}
		\Dbc(\fM)_{\rm{dg}}=\lim_{\mathfrak{U} \stackrel{\alpha}{\to} \mathfrak{M}} 
		\Dbc(\fU)_{\rm{dg}}
	\end{align*}
	where the limit is taken for 
	the $\infty$-category $\iI$ of smooth morphisms
	$\alpha \colon \fU \to \fM$
	where $\fU$ is equivalent to an
	affine derived scheme of the form (\ref{frak:U}). 
	In this paper when we write 
	$\lim_{\mathfrak{U} \stackrel{\alpha}{\to} \mathfrak{M}}(-)$
	for a quasi-smooth $\fM$, 
	the limit is always taken for the $\infty$-category $\iI$
	as above. 

	Following~\cite[Definition~1.1.8]{MR3037900}, 
	a derived stack $\mathfrak{M}$ is called 
	\textit{QCA (quasi-compact and with affine automorphism groups)}
	if the following conditions hold:
	\begin{enumerate}
		\item $\mathfrak{M}$ is quasi-compact;
		\item The automorphism groups of its geometric points are affine;
		\item The classical inertia stack $I_{\mM} \cneq \Delta
		\times_{\mM \times \mM} \Delta$
		is of finite presentation over $\mM$. 
	\end{enumerate}
	Let $\fM$ be a quasi-smooth
	derived stack. 
	We denote by $\Ind \Dbc(\fM)_{\rm{dg}}$ the dg-category of its ind-coherent 
	sheaves (see~\cite{MR3136100})
	\begin{align}\notag
		\Ind \Dbc(\fM)_{\rm{dg}} \cneq \lim_{\fU \stackrel{\alpha}{\to}\fM}
		\Ind \Dbc(\fU)_{\rm{dg}},
	\end{align}
	where the limit is taken for the $\infty$-category $\iI$ 
	as above. 
	The QCA condition will be useful 
	since we have the following theorem:  
	\begin{thm}\emph{(\cite[Theorem~3.3.5]{MR3037900})}\label{thm:QCA}
		If $\fM$ is QCA, 
		then 
		$\Ind \Dbc(\fM)_{\rm{dg}}$ is compactly generated 
		with $(\Ind \Dbc(\fM)_{\rm{dg}})^{\rm{cp}}=\Dbc(\fM)_{\rm{dg}}$. 
		In particular, we have 
		\begin{align}\notag
			\Ind \Dbc(\fM)_{\rm{dg}}=\Ind(\Dbc(\fM)_{\rm{dg}}).
		\end{align} 
	\end{thm}
	
	Let $\mathfrak{M}$ be a quasi-smooth derived stack. 
	We denote by 
	$\Omega_{\mathfrak{M}}[-1]$ the $(-1)$-shifted
	cotangent derived stack of $\mathfrak{M}$
	\begin{align*}
		p \colon 
		\Omega_{\mathfrak{M}}[-1] \cneq 
		\Spec_{\mathfrak{M}}
		S(\mathbb{T}_{\mathfrak{M}}[1]) \to \mathfrak{M}. 
	\end{align*}
	Here $\mathbb{T}_{\mathfrak{M}} \in D^b_{\rm{coh}}(\mathfrak{M})$
	is the tangent complex of $\mathfrak{M}$, which 
	is dual to the cotangent complex 
	$\mathbb{L}_{\mathfrak{M}}$ of $\mathfrak{M}$. 
	The derived stack $\Omega_{\mathfrak{M}}[-1]$ 
	admits a natural $(-1)$-shifted symplectic 
	structure~\cite{MR3090262, Calaque}, 
	which induces the 
	d-critical structure~\cite{MR3399099}
	on its 
	classical truncation $\nN$ 
	\begin{align}\notag
		p_0 \colon \nN \cneq t_0(\Omega_{\mathfrak{M}}[-1])
		\to \mM.
	\end{align}

	Let $\fM_1$, $\fM_2$ be
	quasi-smooth derived stacks 
	with truncations $\mM_i=t_0(\fM_i)$. 
	Let $f \colon \fM_1 \to \fM_2$ be a morphism. 
	Then 
	the morphism 
	$f^{\ast}\mathbb{L}_{\fM_2} \to \mathbb{L}_{\fM_1}$
	induces the diagram
	\begin{align}\notag
		\xymatrix{
			t_0(\Omega_{\fM_1}[-1]) \ar[d] & 
			t_0(\Omega_{\fM_2}[-1]\times_{\fM_2}\fM_1)
			\ar[d] \ar[l]_-{f^{\diamondsuit}} 
			\ar[r]^-{f^{\spadesuit}} 
			\ar@{}[rd]|\square
			& 
			t_0(\Omega_{\fM_2}[-1]) \ar[d] \\
			\mM_1 & \mM_1 \ar@{=}[l] \ar[r]_-{f} & \mM_2. 
		}
	\end{align}
		The morphism $f$ is quasi-smooth if and only
		if $f^{\diamondsuit}$ is a closed immersion, 
		$f$ is smooth if and only if $f^{\diamondsuit}$ is an isomorphism
		(see~\cite[Lemma~3.1.2]{TocatDT}). 
	
Let us take a  conical closed substack
	\begin{align*}
		\zZ \subset \nN=t_0(\Omega_{\mathfrak{M}}[-1]). 
	\end{align*}
	Here $\zZ$ is called conical 
	if it is invariant under the fiberwise $\mathbb{C}^{\ast}$-action on $\nN \to \mM$. 
	Let $\alpha \colon \fU \to \fM$ be a smooth morphism
	such that $\fU$ is of the form (\ref{frak:U}). 
	We have the associated conical closed subscheme
	\begin{align}\notag
		\alpha^{\ast}\zZ \cneq 
		\alpha^{\diamondsuit} (\alpha^{\spadesuit})^{-1}(\zZ) \subset t_0(\Omega_{\mathfrak{U}}[-1])
		=\mathrm{Crit}(w).
	\end{align}
	Here 
	$w$ is given as in (\ref{func:w}). 
As in~\cite{MR3300415}, we define 
		\begin{align}\notag
			\cC_{\zZ, \rm{dg}} \cneq 
			\lim_{\fU \stackrel{\alpha}{\to} \fM}
			\cC_{\alpha^{\ast}\zZ, \rm{dg}}\subset D^b_{\rm{coh}}(\mathfrak{M})_{\rm{dg}}, \ 
			\Ind \cC_{\zZ, \rm{dg}} \cneq \lim_{\fU \stackrel{\alpha}{\to} \fM}
			\Ind \cC_{\alpha^{\ast}\zZ, \rm{dg}}\subset \Ind D^b_{\rm{coh}}(\mathfrak{M})_{\rm{dg}}, 
		\end{align}
	whose homotopy categories are denoted as $\cC_{\zZ}$, $\Ind \cC_{\zZ}$
	respectively. 
	
	\subsection{Definition of $\mathbb{Z}/2$-periodic 
		DT categories}
	For a quasi-smooth and QCA derived stack $\fM$, 
let us take an open substack $\nN^{\rm{ss}}$ and its complement 
	$\zZ$, 
	\begin{align*}
		\nN^{\rm{ss}} \subset \nN, \ 
		\zZ \cneq \nN\setminus \nN^{\rm{ss}}. 
	\end{align*}
In the case that $\nN^{\rm{ss}}$ is $\C$-invariant so that 
$\zZ$ is a conical closed substack, 
the $\C$-equivariant dg or triangulated DT categories were defined in~\cite{TocatDT}
as the Drinfeld or Verdier quotient 
\begin{align*}
	\dDT^{\C}(\nN^{\rm{ss}})_{\rm{dg}} \cneq 
	\Dbc(\fM)_{\rm{dg}}/\cC_{\zZ, \rm{dg}}, \ 
	\dDT^{\C}(\nN^{\rm{ss}}) \cneq \Dbc(\fM)/\cC_{\zZ}. 
	\end{align*}
	The idea of the above definition was based on the Koszul 
	duality equivalence in Theorem~\ref{thm:knoer}, 
which gives an interpretation of the above category as a 
gluing dg-categories of $\C$-equivariant
factorizations. 
	
We give a definition of $\mathbb{Z}/2$-periodic DT categories based 
on Proposition~\ref{prop:Z2:supp} instead of Theorem~\ref{thm:knoer}. 
	For a quasi-smooth derived stack $\fM$, 
	we set
	\begin{align*}
		\fM_{\epsilon} \cneq \fM \times \Spec \mathbb{C}[\epsilon]
		\end{align*}
	where $\deg(\epsilon)=-1$. 
	Then we have 
	\begin{align*}
		t_0(\Omega_{\fM_{\epsilon}}[-1]) =
		t_0(\Omega_{\fM}[-1]) \times 
		t_0(\Omega_{\Spec \mathbb{C}[\epsilon]}[-1])=
		\nN \times \mathbb{A}^1. 
		\end{align*}
	We define 
	\begin{align*}
		\zZ_{\epsilon} \cneq 
		\mathbb{C}^{\ast}(\zZ \times \{1\}) \cup 
		(\nN \times \{0\})
		\subset \nN \times \mathbb{A}^1.
		\end{align*}
	Here $\C$ acts on fibers of $\nN \times \mathbb{A}^1 \to \mM$ 
	by weight two. 
	Note that $\zZ_{\epsilon}$ is a conical closed substack, 
	though $\zZ$ may not be conical. 
	Note that if $\zZ$ is conical, 
	then $\zZ_{\epsilon}=(\zZ \times \mathbb{A}^1) \cup (\nN \times \{0\})$. 
	\begin{defi}\label{def:Z2DT}
		The $\mathbb{Z}/2$-periodic dg or triangulated DT categories 
			$\dDT^{\mathbb{Z}/2}(\nN^{\rm{ss}})_{\rm{dg}}$, 
		$\dDT^{\mathbb{Z}/2}(\nN^{\rm{ss}})$ are defined by 
		Drinfeld or Verdier quotients
		\begin{align*}
			\dDT^{\mathbb{Z}/2}(\nN^{\rm{ss}})_{\rm{dg}} \cneq 
			\Dbc(\fM_{\epsilon})_{\rm{dg}}/\cC_{\zZ_{\epsilon}, \rm{dg}}, \ 
			\dDT^{\mathbb{Z}/2}(\nN^{\rm{ss}})
			\cneq \Dbc(\fM_{\epsilon})/\cC_{\zZ_{\epsilon}}. 
			\end{align*}
		\end{defi}
	If $\fM$ is an affine derived scheme of the form (\ref{frak:U}), 
	then $\dDT^{\mathbb{Z}/2}(\nN^{\rm{ss}})$ is equivalent 
	to the derived category of $\mathbb{Z}/2$-periodic factorizations 
	by Proposition~\ref{prop:Z2:supp}. 
	In general, the triangulated category $\dDT^{\mathbb{Z}/2}(\nN^{\rm{ss}})$ is 
	$\mathbb{Z}/2$-periodic by the following lemma:

	\begin{lem}\label{lem:2period}
		The triangulated category 
		$\dDT^{\mathbb{Z}/2}(\nN^{\rm{ss}})$ is $\mathbb{Z}/2$-periodic. 
		\end{lem}
	\begin{proof}
		Let $i \colon \fM_{\epsilon} \hookrightarrow \fM$ be the
		closed immersion induced by 
		the closed immersion 
		$\Spec \mathbb{C}[\epsilon] \hookrightarrow \Spec \mathbb{C}$. 
		Then for any $\eE \in \Dbc(\fM_{\epsilon})$ there is a distinguished triangle
		\begin{align}\label{dist:E}
			\eE[1] \to i^{\ast}i_{\ast}\eE \to \eE \to \eE[2]. 
			\end{align}
		The morphism $i$ is nothing but the projection 
		$\fM_{\epsilon} \to \fM$, 
		and $i^{\ast}i_{\ast}\eE=i_{\ast}\eE\boxtimes \mathbb{C}[\epsilon]$. 
		Since $\mathbb{C}[\epsilon] \in \cC_{\{0\}}$ in $\Dbc(\Spec \mathbb{C}[\epsilon])$, 
		we have 
		$i^{\ast}i_{\ast}\eE \in \cC_{\nN \times \{0\}}$ by~\cite[Lemma~4.6.4]{MR3300415}. 
		Therefore the morphism $\eE \to \eE[2]$ in the sequence (\ref{dist:E}) is an isomorphism 
		in $\dDT^{\mathbb{Z}/2}(\nN^{\rm{ss}})$. 
		\end{proof}
	
	Let $\wW \subset \mM$ be a closed substack, 
	and 
	$\fM_{\circ} \subset \fM$ be the open substack 
	whose classical truncation is $\mM \setminus \wW$. 
	We have the following closed substack
	\begin{align*}
		\zZ_{\circ} \cneq \zZ \setminus p_0^{-1}(\wW)
		\subset \nN_{\circ} \cneq t_0(\Omega_{\mathfrak{M}_{\circ}}[-1]).
	\end{align*}
	Note that $\nN_{\circ}=\nN \setminus p_0^{-1}(\wW)$, and 
	we have the following open immersion
	\begin{align}\label{open:N}
		\nN_{\circ}^{\rm{ss}} \cneq 
		\nN_{\circ} \setminus \zZ_{\circ} \hookrightarrow
		\nN \setminus \zZ =\nN^{\rm{ss}}. 
	\end{align}
	The following lemma is an analogue of~\cite[Lemma~3.2.9]{TocatDT}:
	\begin{lem}\label{lem:replace0}
		Suppose that $p_0^{-1}(\wW) \subset \zZ$, 
		so that the open immersion (\ref{open:N}) is an isomorphism. 
		Then 
		the restriction functor
		gives an equivalence
		\begin{align*}
			\mathcal{DT}^{\mathbb{Z}/2}
			(\nN^{\rm{ss}}) \stackrel{\sim}{\to}
			\mathcal{DT}^{\mathbb{Z}/2}
			(\nN_{\circ}^{\rm{ss}}). 
		\end{align*}
		\end{lem}
	\begin{proof}
		Since $p_0^{-1}(\wW) \times \mathbb{A}^1 \subset \zZ_{\epsilon}$
			 and 
		$(\zZ_{\circ})_{\epsilon}=\zZ_{\epsilon} \setminus (p_0^{-1}(\wW) \times \mathbb{A}^1)$, 
		we can directly apply~\cite[Lemma~3.2.9]{TocatDT} to obtain the lemma. 
		\end{proof}
	
	\subsection{Comparison with $\C$-equivariant DT categories}
	For a conical closed substack 
	$\zZ \subset t_0(\Omega_{\fM}[-1])$
	with $\nN^{\rm{ss}}=\nN \setminus \zZ$, 
	we have both of $\C$-equivariant DT 
	category $\dDT^{\C}(\nN^{\rm{ss}})$ and 
	$\mathbb{Z}/2$-periodic one 
	$\dDT^{\mathbb{Z}/2}(\nN^{\rm{ss}})$. 
	In this subsection, we compare these 
	categories. 
	We first prove some lemmas. 
	
	\begin{lem}\label{lem:cgen0}
		Let $\fM$ be a quasi-smooth and QCA derived 
		stack and $\zZ \subset t_0(\Omega_{\fM}[-1])$ be a conical 
		closed substack. 
		Then $(\Ind \cC_{\zZ})^{\rm{cp}}=\cC_{\zZ}$. 
	\end{lem}
	\begin{proof}
		The lemma is a proved in the last part of~\cite[Theorem~7.2.2]{TocatDT}. 
		Since we have $\Dbc(\fM)=(\Ind \Dbc(\fM))^{\rm{cp}}$ by~\cite[Theorem~3.3.5]{MR3037900}, 
		we have 
		\begin{align*}
			\cC_{\zZ} \subset (\Ind \Dbc(\fM))^{\rm{cp}} \cap \Ind \cC_{\zZ} \subset 
			(\Ind \cC_{\zZ})^{\rm{cp}}.
		\end{align*}
	As for 
		the converse direction
		$(\Ind \cC_{\zZ})^{\rm{cp}} \subset \cC_{\zZ}$, 
		we take an object $\eE \in (\Ind \cC_{\zZ})^{\rm{cp}}$
		and 
		a smooth morphism $\alpha \colon \fU \to \fM$
		where $\fU$ is an affine derived scheme of the form (\ref{frak:U}). 
		Since
		the pull-back $\alpha^{\ast} \colon \Ind \cC_{\zZ} \to \Ind \cC_{\alpha^{\ast}\zZ}$
		admits a continuous right adjoint 
		$\alpha^{\mathrm{ind}}_{\ast} \colon \Ind \cC_{\alpha^{\ast}\zZ} \to \Ind \cC_{\zZ}$, 
		we have
		$\alpha^{\ast}\eE \in (\Ind \cC_{\alpha^{\ast}\zZ})^{\rm{cp}}=
			\cC_{\alpha^{\ast}\zZ}$. 
			Here the last identity follows 
			from~\cite[Corollary~9.2.7, Corollary~9.2.8]{MR3300415}. 
			Therefore we have 
			$\eE \in \cC_{\zZ}$. 
	\end{proof}
	
	\begin{lem}\label{lem:cgen}
		Let $\fM_i$ for $i=1, 2$ be quasi-smooth and QCA
		derived stacks, 
		and $\zZ_i \subset t_0(\Omega_{\fM_i}[-1])$ 
		be conical closed substacks. 
		Suppose that $\Ind \cC_{\zZ_i}$ are compactly generated. 
		Then the subcategories 
		\begin{align*}
			\Ind \cC_{p_1^{\ast}\zZ_1 \cap p_2^{\ast}\zZ_2}, \ 
			\Ind \cC_{p_1^{\ast}\zZ_1 \cup p_2^{\ast}\zZ_2}
			\subset \Ind \Dbc(\fM_1 \times \fM_2)
		\end{align*}
		are compactly 
		generated. Here 
		$p_i \colon \fM_1 \times \fM_2 \to \fM_i$ is the projection. 
	\end{lem}
	\begin{proof}
		We first show that $\Ind \cC_{\zZ_1 \cap \zZ_2}$ is compactly generated. 
		By the assumption and Lemma~\ref{lem:cgen0}, 
		$\Ind \cC_{\zZ_i}$ are compactly generated 
		with compact objects $\cC_{\zZ_i}$. Therefore 
		the argument of~\cite[Corollary~4.20]{MR3037900}
		shows that 
		\begin{align*}
			\Ind(\cC_{\zZ_1, \rm{dg}} \otimes \cC_{\zZ_2, \rm{dg}}) 
			\simeq 	\lim_{\fU_1 \stackrel{\alpha_1}{\to} \fM_1}
			\lim_{\fU_2 \stackrel{\alpha_2}{\to} \fM_2}
			\Ind (\cC_{\alpha_1^{\ast}\zZ, \rm{dg}} \otimes \cC_{\alpha_1^{\ast}\zZ. \rm{dg}}).
		\end{align*}
		By~\cite[Lemma~4.6.4]{MR3300415}, we have the equivalence 
		\begin{align*}
			\Ind (\cC_{\alpha_1^{\ast}\zZ, \rm{dg}} \otimes \cC_{\alpha_1^{\ast}\zZ, \rm{dg}})
			\stackrel{\sim}{\to} \Ind(\cC_{p_1^{\ast}\alpha_1^{\ast}\zZ_1 \cap p_2^{\ast}\alpha_2^{\ast}\zZ_2, \rm{dg}}). 
		\end{align*}
		It follows that we have the equivalence 
		\begin{align*}
			\Ind \cC_{p_1^{\ast}\zZ_1 \cap p_2^{\ast}\zZ_2, \rm{dg}}
			\simeq \Ind(\cC_{\zZ_1, \rm{dg}} \otimes \cC_{\zZ_2, \rm{dg}}). 
		\end{align*}
		So $\Ind \cC_{p_1^{\ast}\zZ_1 \cap p_2^{\ast}\zZ_2}$ is generated by compact objects 
		$\cC_{\zZ_1} \boxtimes  \cC_{\zZ_2}$. Then the 
		compact generation for $\Ind \cC_{p_1^{\ast}\zZ_1 \cup p_2^{\ast}\zZ_2}$
		holds since it is generated by 
		$\Ind \cC_{p_1^{\ast}\zZ_1}$ and $\Ind \cC_{p_2^{\ast}\zZ_2}$, and the latter categories are 
		compactly generated by the above argument. 
	\end{proof}

\begin{lem}\label{lem:Zblcp}
	In the setting of Lemma~\ref{lem:cgen0}, 
	the subcategory $\Ind \cC_{\zZ_{\epsilon}} \subset \Ind \Dbc(\fM_{\epsilon})$
	is compactly generated. 
	\end{lem}
	\begin{proof}
		Since $\zZ_{\epsilon}=(\zZ \times \mathbb{A}^1) \cup (\nN \times \{0\})$ and 
		$\Ind \cC_{\{0\}} \subset \Ind \Dbc(\bullet_{\epsilon})$
		is compactly generated by $\oO_0=\mathbb{C}[\epsilon]/(\epsilon)$, 
		the lemma follows from Lemma~\ref{lem:cgen}. 
		\end{proof}
	
	\begin{lem}\label{lem:DTind}
		In the setting of Lemma~\ref{lem:Zblcp}, we have an equivalence
		\begin{align*}
			\Ind \dDT^{\mathbb{Z}/2}(\nN^{\rm{ss}})_{\rm{dg}} \stackrel{\sim}{\to}
			\lim_{\fU \stackrel{\alpha}{\to} \fM}
			\Ind\left( \Dbc(\fU_{\epsilon})_{\rm{dg}}/\cC_{\alpha_{\epsilon}^{\ast}\zZ_{\epsilon}, \rm{dg}}  \right). 
			\end{align*}
		Here $\alpha_{\epsilon} =\alpha \times \id \colon \fU_{\epsilon} \to \fM_{\epsilon}$. 
		\end{lem}
	\begin{proof}
		By Lemma~\ref{lem:cgen0} and Lemma~\ref{lem:Zblcp}, the category 
		$\Ind \cC_{\zZ_{\epsilon}}$ is compactly generated with 
		compact objects $\cC_{\zZ_{\epsilon}}$. Then 
		the lemma follows from~\cite[Proposition~3.2.7]{TocatDT}. 
		\end{proof}
	Let $\mathbb{C}[u^{\pm 1}]$ be 
	the dg algebra with $\deg(u)=2$ and zero differential. 
	We regard it as a dg category with one object. 
	The following result implies that one can 
	recover $\mathbb{Z}/2$-periodic DT category 
	from the $\C$-equivariant one up to idempotent completion. 
	
	\begin{thm}\label{thm:compareDT}
		Let $\fM$ be a 
		quasi-smooth and QCA derived stack and 
		$\zZ \subset t_0(\Omega_{\fM}[-1])$ be a conical 
		closed substack. 
		Suppose that $\Ind \cC_{\zZ}$ is compactly generated. 
		Then there is an equivalence
		\begin{align}\label{equiv:DT20}
			\Ind \mathcal{DT}^{\mathbb{Z}/2}(\nN^{\rm{ss}})_{\rm{dg}}
			\simeq \dR \underline{\hH om}(\mathbb{C}[u^{\pm 1}], 
			\Ind \mathcal{DT}^{\C}(\nN^{\rm{ss}})_{\rm{dg}}),	
		\end{align}
		which restricts to the equivalence 
		\begin{align}\label{equiv:DT2}
			\overline{\mathcal{DT}}^{\mathbb{Z}/2}(\nN^{\rm{ss}})_{\rm{dg}}
			\simeq \dR \underline{\hH om}(\mathbb{C}[u^{\pm 1}], 
			\Ind \mathcal{DT}^{\C}(\nN^{\rm{ss}})_{\rm{dg}})^{\rm{cp}}. 	
		\end{align}
		\end{thm}
	\begin{proof}
		Similarly to Lemma~\ref{lem:DTind}, 
		we have the equivalence (see~\cite[Proposition~3.2.7]{TocatDT})
			\begin{align*}
				\Ind \dDT^{\C}(\nN^{\rm{ss}})_{\rm{dg}} \stackrel{\sim}{\to}
		\lim_{\fU \stackrel{\alpha}{\to} \fM}
		\Ind\left( \Dbc(\fU)_{\rm{dg}}/\cC_{\alpha^{\ast}\zZ, \rm{dg}}  \right). 
		\end{align*}
		Therefore the 
		right hand side of (\ref{equiv:DT20}) is 
		\begin{align*}
			\lim_{\fU \stackrel{\alpha}{\to} \fM}
			\dR \underline{\hH om}(\mathbb{C}[u^{\pm 1}], 
			\Ind\left( \Dbc(\fU)_{\rm{dg}}/\cC_{\alpha^{\ast}\zZ, \rm{dg}} \right) )
			\simeq 	\lim_{\fU \stackrel{\alpha}{\to} \fM}
			\Ind \left( (\Dbc(\fU)_{\rm{dg}}/\cC_{\alpha^{\ast}\zZ, \rm{dg}}) \otimes 
			\mathbb{C}[u^{\pm 1}]   \right). 
			\end{align*}
By Theorem~\ref{thm:knoer}, there is an equivalence
\begin{align*}
\MF^{\C}_{\coh}(\C, 0)_{\rm{dg}} \stackrel{\sim}{\to}\Dbc(\bullet_{\epsilon})_{\rm{dg}}/\cC_{\{0\}, \rm{dg}}
\end{align*}
where the $\C$-action on $\C$ in the left hand side is of weight two. 
The above functor sends 
$(\oO_{\C}, d_{\oO_{\C}}=0)$ to $\oO_0=\mathbb{C}[\epsilon]/(\epsilon)$, 
so we have the quasi-isomorphism
of dg-algebra
\begin{align*}
	\mathbb{C}[u^{\pm 1}] \stackrel{\sim}{\to} 
\Hom^{\ast}_{\Dbc(\bullet_{\epsilon})_{\rm{dg}}/\cC_{\{0\}, \rm{dg}}}(\oO_0, \oO_0). 
\end{align*}
Therefore there exists a natural functor
\begin{align}\label{fun:Cu}
(\Dbc(\fU)_{\rm{dg}}/\cC_{\alpha^{\ast}\zZ, \rm{dg}}) \otimes \mathbb{C}[u^{\pm 1}]
\to \Dbc(\fU_{\epsilon})_{\rm{dg}}/\cC_{\alpha_{\epsilon}^{\ast}\zZ_{\epsilon}, \rm{dg}}
\end{align}
sending $\eE$ to $\eE \boxtimes \oO_0$. 
Under the equivalence (\ref{Psi:supp2}), 
the above functor corresponds to 
the functor 
\begin{align*}
\MF^{\C}_{\coh}(V^{\vee} \setminus \alpha^{\ast}Z, w)_{\rm{dg}}\otimes \mathbb{C}[u^{\pm 1}]
\to \MF_{\coh}^{\C}((V^{\vee} \setminus \alpha^{\ast}Z)\times_{\C}\C, w_{\epsilon})_{\rm{dg}}
\end{align*}
sending $\pP$ to $\pP\boxtimes \oO_{\C}$. 
The above functor is fully-faithful by~\cite[Lemma~3.52]{MR3270588}, 
so the functor (\ref{fun:Cu}) is also fully-faithful. 
By taking the ind-completion, we obtain the fully-faithful 
functor
\begin{align}\label{fun:Cu2}
\Ind\left((\Dbc(\fU)_{\rm{dg}}/\cC_{\alpha^{\ast}\zZ, \rm{dg}}) \otimes \mathbb{C}[u^{\pm 1}] \right)
\to \Ind \left(\Dbc(\fU_{\epsilon})_{\rm{dg}}/\cC_{\alpha_{\epsilon}^{\ast}\zZ_{\epsilon}, \rm{dg}}\right). 
\end{align}
Since $\Ind \Dbc(\fU_{\epsilon})$ is generated by 
exterior products by~\cite[Proposition~6.3.4]{MR3701352}, 
and $\Dbc(\bullet_{\epsilon})$ is generated by 
$\oO_0$, 
the functor (\ref{fun:Cu2}) is essentially surjective, 
hence it is an equivalence. 
By taking the limit of the equivalence (\ref{fun:Cu2}) 
for all $\alpha \colon \fU \to \fM$
and using Lemma~\ref{lem:DTind}, 
we obtain the equivalence (\ref{equiv:DT20}). 

Any object in the subcategory
\begin{align*}
\dDT^{\mathbb{Z}/2}(\nN^{\rm{ss}})_{\rm{dg}} \subset \Ind 
\dDT^{\mathbb{Z}/2}(\nN^{\rm{ss}})_{\rm{dg}}
\end{align*}
is compact by~\cite[Proposition~3.2.7, Lemma~7.1.2]{TocatDT}, 
which generates $\Ind \dDT^{\mathbb{Z}/2}(\nN^{\rm{ss}})_{\rm{dg}}$
by Theorem~\ref{thm:QCA}.
Therefore we have 
\begin{align*}
\overline{\dDT}^{\mathbb{Z}/2}(\nN^{\rm{ss}})_{\rm{dg}} =
\left(  \Ind \dDT^{\mathbb{Z}/2}(\nN^{\rm{ss}})_{\rm{dg}} \right)^{\rm{cp}}.
\end{align*} 
By combining with the equivalence (\ref{equiv:DT20}), 
we obtain the equivalence (\ref{equiv:DT2}). 
		\end{proof}
	
	\section{$\mathbb{Z}/2$-periodic DT categories for local surfaces}
	In this section, we define 
	$\mathbb{Z}/2$-periodic DT categories for local 
	surfaces following basic model in Definition~\ref{def:Z2DT}.
	We then use Theorem~\ref{thm:compareDT} and results
	in~\cite{TocatDT} to prove some wall-crossing equivalences
	for $\mathbb{Z}/2$-periodic DT categories for
	one dimensional stable sheaves on local surfaces. 
	by push-forward. 
	\subsection{Derived moduli stacks of coherent sheaves on surfaces}
	Let $S$ be a smooth projective surface over $\mathbb{C}$. 
	We consider the 
	derived Artin stack
	\begin{align}\notag
		\mathfrak{Perf}_S \colon dAff^{op} \to SSets
	\end{align}
	which sends an affine derived scheme 
	$T$ to the $\infty$-groupoid of 
	perfect complexes on $T \times S$, 
	constructed in~\cite{MR2493386}. 
	We have the open substack
	\begin{align*}
		\mathfrak{M}_S \subset \mathfrak{Perf}_S
	\end{align*}
	corresponding to perfect complexes on $S$ 
	quasi-isomorphic to 
	coherent sheaves on $S$. 
	Since
	any object in $\Coh(S)$ is perfect as $S$ is smooth,
	the derived Artin stack $\mathfrak{M}_S$ is 
	the derived  
	moduli stack of objects in 
	$\Coh(S)$. 
	It is well-known that $\fM_S$ is a quasi-smooth derived Artin stack (see~\cite[Section~3.4]{TocatDT}).
	
	Let $N(S)$ be the numerical Grothendieck group of 
	$S$
	\begin{align}\notag
		N(S) \cneq K(S)/\equiv
	\end{align}
	where $F_1, F_2 \in K(S)$
	satisfies $F_1 \equiv F_2$ if $\ch(F_1)=\ch(F_2)$. 
	Then $N(S)$ is a finitely generated free abelian 
	group. 
	We have the decompositions into open and closed substacks
	\begin{align*}
		\mathfrak{M}_S=\coprod_{v \in N(S)} \mathfrak{M}_S(v), \ 
		\mM_S=\coprod_{v \in N(S)} \mM_S(v)
	\end{align*}
	where each component 
	corresponds to sheaves $F$
	with $[F]=v$. 
	
	Note that the automorphism 
	group of a sheaf $F$ on $S$
	contains a one dimensional torus 
	$\mathbb{C}^{\ast} \subset \Aut(F)$
	given by the scalar multiplication. 
	Therefore 
	the inertia stack $I_{\fM_S}$ of $\fM_S$
	admits an embedding 
	$(\mathbb{C}^{\ast})_{\fM_S} \subset I_{\fM_S}$. 
	As in~\cite[Subsection~3.2.4]{TocatDT}, 
	we have the 
	$\mathbb{C}^{\ast}$-rigidification 
	\begin{align}\label{C:rig}
		\mathfrak{M}_S(v) \to 
		\mathfrak{M}_S^{\mathbb{C}^{\ast}\mathchar`-\mathrm{rig}}(v), 
	\end{align}
which rigidifies the above $\C$-automorphisms, i.e. 
the automorphism group of $F$ in $\mathfrak{M}_S^{\mathbb{C}^{\ast}\mathchar`-\mathrm{rig}}(v)$
is $\Aut(F)/\C$.

	\subsection{Moduli stacks of compactly supported sheaves on local 
		surfaces}
	For a smooth projective surface $S$, 
	we consider its total space of the canonical line bundle: 
	\begin{align*}
		X=\mathrm{Tot}_S(\omega_S) \stackrel{\pi}{\to}S.
	\end{align*}
	Here $\pi$ is the projection. 
	We denote by 
	$\Coh_{\rm{cpt}}(X) \subset \Coh(X)$
	the subcategory of compactly supported coherent 
	sheaves on $X$. 
	We consider the 
	classical Artin stack
	\begin{align*}
		\mM_X : 
		Aff^{\rm{op}} \to Groupoid
	\end{align*}
	whose $T$-valued points for $T \in Aff$
	form the groupoid of $T$-flat families of 
	objects in 
	$\Coh_{\rm{cpt}}(X)$. 
	We have the decomposition into open and closed 
	substacks
	\begin{align*}
		\mM_X=\coprod_{v\in N(S)} \mM_X(v)
	\end{align*}
	where $\mM_X(v)$ corresponds to compactly 
	supported
	sheaves $F$ on $X$ with $[\pi_{\ast}F]=v$. 
	By pushing forward to $S$, we have the natural 
	morphism 
	\begin{align}\label{pi:ast}
		\pi_{\ast} \colon \mM_X(v) \to \mM_S(v), \ 
		F \mapsto \pi_{\ast}F. 
	\end{align}
	Moreover we have an isomorphism of stacks over $\mM_S(v)$ (see~\cite[Lemma~3.4.1]{TocatDT})
	\begin{align}\label{isom:etaM}
		\eta \colon 
		\mM_X(v) \stackrel{\cong}{\to} t_0(\Omega_{\mathfrak{M}_S(v)}[-1]). 
	\end{align}
	Similarly to (\ref{C:rig}), (\ref{isom:etaM}), 
	we also have the $\mathbb{C}^{\ast}$-rigidification
	and the isomorphism 
	\begin{align}\label{isom:rigid:S}
		\mM_X(v) \to \mM_X(v)^{\C\rig}, \ 
		\mM_X(v)^{\mathbb{C}^{\ast}\mathchar`-\mathrm{rig}} \stackrel{\cong}{\to}
		t_0
		(\Omega_{\mathfrak{M}_S(v)^{\C\rig}}[-1]).
	\end{align}

	\subsection{Definition of $\mathbb{Z}/2$-periodic 
		DT categories for local surfaces}\label{subsec:catDT:sigma}
	Let $A(S)_{\mathbb{C}}$ be the ample cone
	\begin{align*}
		A(S)_{\mathbb{C}} \cneq \{B+iH : H \mbox{ is ample}\}
		\subset \mathrm{NS}(S)_{\mathbb{C}}. 
	\end{align*}
	Then any element $\sigma=B+iH \in A(S)_{\mathbb{C}}$ 
	determines the $\sigma$-stability on $\Coh_{\rm{cpt}}(X)$ in 
	a standard way (see~\cite[Definition~3.4.3]{TocatDT}). 
	We have the open substacks
	\begin{align}\notag
		\mM_{X}^{\sigma \st}(v)
		\subset 
		\mM_{X}^{\sigma}(v)
		\subset \mM_X(v)
	\end{align}
	corresponding to $\sigma$-stable sheaves, $\sigma$-semistable sheaves, 
	respectively. 
	By the GIT construction of 
	the moduli stack $\mM_X^{\sigma}(v)$, it admits a good moduli 
	space (cf.~\cite[Section~1.2]{MR3237451})
	\begin{align}\label{gmoduli:M}
		\mM_X^{\sigma}(v) \to M_X^{\sigma}(v)
	\end{align}
	where $M_X^{\sigma}(v)$ is a quasi-projective 
	scheme whose closed points correspond to $\sigma$-polystable objects. 
When we have the equality 
$\mM_{X}^{\sigma \st}(v)
		=
		\mM_{X}^{\sigma}(v)$, 
		then 
	the morphism (\ref{gmoduli:M}) is a $\C$-gerbe
	so that 
	\begin{align*}
		M_{X}^{\sigma}(v) =
		\mM_{X}^{\sigma}(v)^{\mathbb{C}^{\ast} \rig}
	\end{align*}
	holds. 
	Note that it 
	admits a  
	$\mathbb{C}^{\ast}$-action induced by the fiberwise
	weight two  
	$\mathbb{C}^{\ast}$-action on $\pi \colon X \to S$. 
	
	The moduli stack of $\sigma$-semistable sheaves 
	$\mM_X^{\sigma}(v)$ is of finite 
	type, while $\mathfrak{M}_S(v)$ is not quasi-compact in general, so 
	in particular it is not QCA. 
	So we take a quasi-compact open derived substack
	$\mathfrak{M}_{S}(v)_{\rm{qc}} \subset \mathfrak{M}_S(v)$
	satisfying the condition  
	\begin{align}\label{take:open}
		\pi_{\ast}(\mM_{X}^{\sigma}(v)) \subset t_0(\mathfrak{M}_{S}(v)_{\rm{qc}}). 
	\end{align}
	Here $\pi_{\ast}$ is the morphism (\ref{pi:ast}). 
	Note that $\fM_S(v)_{\rm{qc}}$ is QCA. 
	By the isomorphism $\eta$ in (\ref{isom:etaM}), 
	we have the conical closed substack
	\begin{align*}
		\zZ^{\sigma\mathchar`-\rm{us}}
		\cneq t_0(\Omega_{\mathfrak{M}_S(v)_{\rm{qc}}}[-1])
		\setminus \eta(\mM_{X}^{\sigma}(v)) \subset t_0(\Omega_{\mathfrak{M}_S(v)_{\rm{qc}}}[-1]).
	\end{align*}
	By taking the $\mathbb{C}^{\ast}$-rigidification
	and  
	using the isomorphism (\ref{isom:rigid:S}), 
	we also have the conical closed substack
	\begin{align*}
		(\zZ^{\sigma \us})^{\mathbb{C}^{\ast} \rig}
		\subset t_0(\Omega_{\mathfrak{M}_S(v)^{\C\rig}_{\rm{qc}}}[-1]).
	\end{align*}
	We apply Definition~\ref{def:Z2DT}
	for $\fM=\mathfrak{M}_S(v)^{\C\rig}_{\rm{qc}}$
	and $\zZ=(\zZ^{\sigma \us})^{\mathbb{C}^{\ast} \rig}$
	to give the following definition:  
	\begin{defi}\label{def:period}
		Suppose that $\mM_X^{\sigma\st}(v)=\mM_X^{\sigma}(v)$ holds. 
		The $\mathbb{Z}/2$-periodic dg or triangulated 
		DT categories for $M_X^{\sigma}(v)$
		\begin{align}\label{def:DTM}
			\mathcal{DT}^{\mathbb{Z}/2}(M_X^{\sigma}(v))_{\dg}, \ 
			\dDT^{\mathbb{Z}/2}(M_X^{\sigma}(v))
		\end{align}
		are defined by setting 
		$\fM=\mathfrak{M}_S(v)^{\C\rig}_{\rm{qc}}$
		and $\zZ=(\zZ^{\sigma \us})^{\mathbb{C}^{\ast} \rig}$
		in Definition~\ref{def:Z2DT}. 
			\end{defi}
		\begin{rmk}\label{rmk:indepe}
			By Lemma~\ref{lem:replace0}, the categories (\ref{def:DTM})
			are independent of a choice of $\fM_S(v)_{\rm{qc}}$
			satisfying the condition (\ref{take:open}) up to equivalence. 
			\end{rmk}
	
	\subsection{Categorical wall-crossing for moduli spaces of one dimensional sheaves}
	We denote by 
	\begin{align*}
		\Coh_{\le 1}(X) \subset \Coh_{\rm{cpt}}(X)
	\end{align*}
	the  
	abelian subcategory consisting 
	of sheaves $F$ with $\dim \Supp(F) \le 1$. 
	We also denote by $N_{\le 1}(S)$ the subgroup of $N(S)$
	spanned by 
	sheaves $F \in \Coh_{\le 1}(S)$. 
	Note that we have an isomorphism
	\begin{align}\notag
		N_{\le 1}(S) \stackrel{\cong}{\to} 
		\mathrm{NS}(S) \oplus \mathbb{Z}, \ 
		F \mapsto (l(F), \chi(F))
	\end{align}
	where $l(F)$ is the fundamental one cycle of $F$. 
	Below we identify an element $v \in N_{\le 1}(S)$
	with $(\beta, n) \in \mathrm{NS}(S) \oplus \mathbb{Z}$
	by the above isomorphism. 
	We write $[F]=(\beta, n)$ for 
	$F \in \Coh_{\le 1}(S)$ by the above isomorphism. 
	For $\sigma=B+iH \in A(S)_{\mathbb{C}}$, 
	we set 
	\begin{align*}
			\mu_{\sigma}(F) \cneq
			\frac{n-B \cdot \beta}{H \cdot \beta} \in \mathbb{Q} \cup \{\infty\}. 
		\end{align*}
	Then $F$ is 
	$\sigma$-(semi)stable
	if and only
	for any subsheaf $0 \neq F' \subsetneq F$ we have 
	the inequality 
	$\mu_{\sigma}(F')<(\le) \mu_{\sigma}(F)$. 
	
	We fix a primitive element 
	$v=(\beta, n) \in N_{\le 1}(S)$
	such that $\beta>0$. 
	Here we write $\beta>0$ if $\beta=[C]$ for a non-zero 
	effective divisor $C$ on $S$. 
	For each decomposition 
	\begin{align*}
		v=v_1+v_2, \ 
		v_i=(\beta_i, n_i), \
		\beta_i>0
	\end{align*}
	we define 
	\begin{align*}
		W_{v_1, v_2}& \cneq \{\sigma \in A(S)_{\mathbb{C}} : 
		\mu_{\sigma}(v_1)=\mu_{\sigma}(v_2)\} \\
		&=\{ B+iH \in A(S)_{\mathbb{C}} : 
		(n_1 \beta_2-n_2 \beta_1) \cdot H=B\beta_1 \cdot H\beta_2
		-B\beta_2 \cdot H\beta_1\}. 
	\end{align*}
	Since $v$ is primitive, 
	$W_{v_1, v_2}\subsetneq A(S)_{\mathbb{C}}$ and 
	$W_{v_1, v_2}$ is a real codimension one 
	hypersurface in $A(S)_{\mathbb{C}}$. 
	For a fixed $v$, the set of 
	hypersurfaces $W_{v_1, v_2}$ are called 
	\textit{walls}. 
	It is easy to see that the walls are locally 
	finite. Also each connected component 
	\begin{align*}
		\cC \subset A(S)_{\mathbb{C}} \setminus \bigcup_{v_1+v_2=v}
		W_{v_1, v_2}
	\end{align*}
	is called a \textit{chamber}. 
	From the construction of walls, the moduli stacks 
	$\mM_{S}^{\sigma}(v)$, $\mM_{X}^{\sigma}(v)$ are constant 
	if $\sigma$ is contained in a chamber, but may change 
	when $\sigma$ crosses a wall. 
	Moreover if $\sigma$ lies in a chamber, 
	they consist of $\sigma$-stable sheaves. 
	
	Let us take 
	$\sigma \in A(S)_{\mathbb{C}}$ which lies on a wall 
	and take $\sigma_{\pm} \in A(S)_{\mathbb{C}}$ which 
	lie on its adjacent chambers. 
	Then by~\cite[Theorem~8.3]{Toddbir}, 
	the wall-crossing diagram
	\begin{align*}
		\xymatrix{
			M_X^{\sigma_+}(v) \ar[rd] & & \ar[ld] M_X^{\sigma_-}(v) \\
			& M_X^{\sigma}(v) &
		}
	\end{align*}
	is a \textit{d-critical flop}, which 
	is a d-critical analogue of usual flops in birational geometry.  
	Therefore 
	following the discussion in~\cite[Subsection~1.1.2]{TocatDT},  
	we propose the following conjecture: 
	\begin{conj}\label{conj1}
		There exists an equivalence of $\mathbb{Z}/2$-periodic
		DT categories:
		\begin{align}\notag
			\dD\tT^{\mathbb{Z}/2}(M_{X}^{\sigma_+}(v)) \stackrel{\sim}{\to}
			\dD\tT^{\mathbb{Z}/2}(M_{X}^{\sigma_-}(v)).
		\end{align}
		In particular $\dDT^{\mathbb{Z}/2}(M_X^{\sigma}(v))$ 
		is independent of a choice of 
		generic $\sigma$ up to equivalence. 
	\end{conj}
	The result of Theorem~\ref{thm:compareDT} reduces the 
	above conjecture (up to idempotent completion) to the case of $\C$-equivariant 
	version, which was proved in~\cite{TocatDT} in several cases. 
	Indeed we have the following: 
	\begin{thm}\label{thm:period2}
		In the setting of Conjecture~\ref{conj1}, suppose that the 
		following condition holds
		\begin{align*}
			\mM_X^{\sigma}(v) \subset \pi_{\ast}^{-1}(\mM_S^{\sigma}(v)). 
		\end{align*}
		Then we have an equivalence
		\begin{align}\label{equiv:period2}
			\overline{\dD\tT}^{\mathbb{Z}/2}(M_{X}^{\sigma_+}(v)) \stackrel{\sim}{\to}
			\overline{\dD\tT}^{\mathbb{Z}/2}(M_{X}^{\sigma_-}(v)).
		\end{align}
	\end{thm}
	\begin{proof}
		It is proved in~\cite[Theorem~1.4.4]{TocatDT}
		that there is an equivalence 
		\begin{align*}
			\dD\tT^{\C}(M_{X}^{\sigma_+}(v)) \stackrel{\sim}{\to}
			\dD\tT^{\C}(M_{X}^{\sigma_-}(v)).
		\end{align*}
		The above equivalence is given by windows, i.e. 
		there exists a triangulated subcategory 
		$\wW \subset \Dbc(\fM_S^{\sigma}(v))$ such that the composition 
		functors 
		\begin{align*}
			\wW \to \Dbc(\fM_S^{\sigma}(v)) \to \dD \tT^{\C}(\fM_S^{\sigma_{\pm}}(v))
		\end{align*}
		are equivalences. 
		The right functors are just quotient functors, so 
		they lift to dg functors 
		\begin{align}\label{quot:dg}
			\Dbc(\fM_S^{\sigma}(v))_{\dg} \to \dD \tT^{\C}(\fM_S^{\sigma_{\pm}}(v))_{\dg}.
		\end{align}
		Let $\wW_{\dg} \subset \Dbc(\fM_S^{\sigma}(v))_{\dg}$ be the 
		full dg-subcategory whose objects are isomorphic to objects in $\wW$
		in the homotopy category. 
		By composing $\wW_{\dg} \subset \Dbc(\fM_S^{\sigma}(v))_{\dg}$
		with dg-functors (\ref{quot:dg}), we
		have dg-functors  
		\begin{align*}
			\dDT^{\C}(\fM_S^{\sigma_+}(v))_{\dg}
			\leftarrow \wW_{\dg} 
			\to
			\dDT^{\C}(\fM_S^{\sigma_+}(v))_{\dg}
		\end{align*}
	which induce equivalences on homotopy categories, i.e. 
	equivalences of dg-categories in the convention of 
	Subsection~\ref{subsec:notation}. 
		By taking ind-completions, applying the inner homomorphism  
		of dg-categories
		$\dR \underline{\hH om}(\mathbb{C}[u^{\pm 1}], -)$ and 
		taking the subcategories of compact objects, we obtain an equivalence 
	by Theorem~\ref{thm:compareDT}
		\begin{align*}
			\overline{\dD\tT}^{\mathbb{Z}/2}(M_{X}^{\sigma_+}(v))_{\dg} \simeq
			\overline{\dD\tT}^{\mathbb{Z}/2}(M_{X}^{\sigma_-}(v))_{\dg}. 	
		\end{align*}
		Therefore we obtain the equivalence (\ref{equiv:period2}). 
	\end{proof}
	
	\begin{rmk}\label{rmk:DK}
		In~\cite{TocatDT}, we proved 
		several other wall-crossing equivalences or 
		fully-faithful functors between 
		$\C$-equivariaint DT categories under 
		d-critical flops or d-critical flips. 
		By the same argument of Theorem~\ref{thm:period2}, 
		we can also prove $\mathbb{Z}/2$-periodic 
		version of these equivalences or fully-faithful functors. 
		
		For example in~\cite[Theorem~5.5.5]{TocatDT}, we proved the 
		existence of a fully-faithful functor under 
		MNOP/PT wall-crossing for reduced curve classes. 
		Namely let 
		$I_n(X, \beta)$ and $P_n(X, \beta)$ be 
		MNOP and PT moduli spaces on the local surface $X$
		respectively (see~\cite[Section~4.2]{TocatDT}). 
		Similarly to Definition~\ref{def:period}, we
		can define $\mathbb{Z}/2$-periodic 
		MNOP and PT categories
		\begin{align*}
			\dDT^{\mathbb{Z}/2}(I_n(X, \beta)), \ 
			\dDT^{\mathbb{Z}/2}(P_n(X, \beta))
			\end{align*}
		following Definition~\ref{def:Z2DT}. 
		Since $I_n(X, \beta) \dashrightarrow P_n(X, \beta)$
		is a d-critical flip, we expect the existence of 
		a fully-faithful functor 
		\begin{align}\label{FF:IP}
			\dDT^{\mathbb{Z}/2}(P_n(X, \beta))
			\hookrightarrow \dDT^{\mathbb{Z}/2}(I_n(X, \beta)). 
			\end{align}
		The argument of Theorem~\ref{thm:period2}
		also proves the existence of a fully-faithful functor 
		(\ref{FF:IP})
		when $\beta$ is a reduced curve class, up to idempotent completions. 
				\end{rmk}
	
	\bibliographystyle{amsalpha}
	\bibliography{math}

	\vspace{5mm}
	
	Kavli Institute for the Physics and 
	Mathematics of the Universe (WPI), University of Tokyo,
	5-1-5 Kashiwanoha, Kashiwa, 277-8583, Japan.

	\textit{E-mail address}: yukinobu.toda@ipmu.jp
	\end{document}